\documentclass{tac}

\usepackage{xy}

\usepackage{amsfonts}

\usepackage{amssymb}

\usepackage{amsmath}

\input diagxy

\author {Marco Grandis and Robert Par\'e}

\thanks{Work supported by a research contract of the University of Genoa.}

\address{Dipartimento di Matematica, Universit\`a di Genova. Via Dodecaneso 35, 16146-Genova, Italy\\Department of Mathematics and Statistics, Dalhousie University,
 Halifax, NS, Canada, B3H 4R2}

\title {Intercategories}

\copyrightyear{2014}


\eaddress{grandis@dima.unige.it\\pare@mathstat.dal.ca}

\def\ic{\bf\sf}

\mathrmdef{id}
\mathrmdef{inj}
\mathrmdef{Id}
\mathrmdef{Lan}

\def\Doub{{\mathbb D}{\rm bl}}

\def\Dlax{{\cal L}{\it x}{\cal D}{\it bl}}

\def\Dcolax{{\cal C}{\it x}{\cal D}{\it bl}}
\mathrmdef{Ob}

\def\Cat{\mbox{{$\cal C$}\it at}}
\def\CAT{\mbox{{$\cal CAT$}}}

\newcommand{\ov}[1]{{\bar{#1}}}
\newcommand{\ovv}[1]{{\tilde{#1}}}

\newcommand{\todd}[2]{\xymatrix@1{\ar[r]|\bb^{#1}_{#2}&}}
\newcommand{\todo}[1]{\xymatrix@1{\ar[r]|\bb^{#1}&}}
\newcommand{\todu}[1]{\xymatrix@1{\ar[r]|\bb_{#1}&}}
\newcommand{\todol}[1]{\xymatrix@1@C=40pt{\ar[r]|\bb^{#1}&}}

\newcommand{\tod}{\xymatrix@1{\ar[r]|\bb&}}

\newcommand{\toc}{\xymatrix@1{\ar[r]|\cc&}}

\newtheoremrm{definition}{Definition}
\newtheorem{theorem}{Theorem}
\newtheorem{proposition}{Proposition}

\newtheoremrm{remark}{Remark}
\newtheoremrm{remarks}{Remarks}
\newtheoremrm{example}{Example}
\newtheoremrm{notation}{Notation}
\newtheoremrm{``proof''}{``Proof''}
\newtheoremrm{note}{Note}
\newtheoremrm{examples}{Examples}

\input xy
\xyoption{all}

\mathrmdef{Hom}
\mathrmdef[colim]{\varinjlim\,}

\mathbfdef{Set}

\def\Cat{{\cal C}{\it at}}

\newbox\bbox
\setbox\bbox = \hbox to 0pt{\hss $\scriptstyle\bullet$\hss}
\def\bb{\usebox{\bbox}}

\newbox\cbox
\setbox\cbox = \hbox to 0pt{\hss $\circ$\hss}
\def\cc{\usebox{\cbox}}

\begin{document}

\maketitle

\begin{abstract} We introduce a $3$-dimensional categorical structure which we call intercategory. This is a kind of weak triple category with three kinds of arrows, three kinds of $2$-dimensional cells and one kind of $3$-dimensional cells. In one dimension, the compositions are strictly associative and unitary, whereas in the other two, these laws only hold up to coherent isomorphism. The main feature is that the interchange law between the second and third compositions does not hold, but rather there is a non-invertible comparison cell which satisfies some coherence conditions. We introduce appropriate morphisms of intercategory, of which there are three types, and cells relating these. We show that these fit together to produce a strict triple category of intercategories.

\end{abstract}


\section*{Introduction}

With this paper, we undertake the study of intercategories (short for interchange categories). These are special kinds of three-dimensional categories. Roughly speaking, they consist of two double categories that share the same horizontal structure, and whose operations are related by interchange morphisms.

A triple category \cite{E} is a category object in double categories which themselves are category objects in ${\bf Cat}$. An analysis of this reveals that there are three kinds of arrows each with its own category structure, and for any two of these, appropriate double cells with two compositions producing a double category. Finally there are triple cells which look like blocks that can be pasted in three directions, satisfying interchange for any two of them. The generalization to higher dimensions is straightforward.

In practice this is too rigid. Already for double categories, where the main examples involve the composition of spans or profunctors of some sort, there is a need to relax the associativity and unit laws to allow isomorphisms. Thus for a weak double category there is one direction in which the equations hold strictly and a second one in which they hold up to coherent isomorphism, just like for bicategories. Any attempt to allow the equations to hold up to isomorphism in both directions simultaneously only leads to a vicious circle from which there seems to be no escape. Yet it would be nice to be able to do this. For example, transposing the two directions is a useful duality for (strict) double categories but is not available for weak ones.

The quintet construction is another case in point. From a $2$-category we get a double category whose horizontal and vertical arrows are the $1$-cells of the $2$-category, and whose double cells are squares with a $2$-cell in them. This is a nice ``symmetric'' way of turning a $2$-category into a double one and is in fact the universal solution to providing companions for every $1$-cell \cite{GP}. The quintet construction works perfectly well for bicategories. It just doesn't give a weak double category. But it almost does. 

These considerations prompted Verity to introduce what he called double bicategories \cite{V}. The name may suggest that this is a four-dimensional concept although the intention is simply a slight relaxation of the double category notion. As we shall see in \cite{PartII}, it is really a three-dimensional concept, an intercategory in fact. So intercategories may also be viewed as weakened double categories held together by a more basic strict structure.

There are many ways of weakening the notion of triple category. The theory of weak $n$-categories would suggest having composition in the first direction satisfy the category equations strictly, in the second direction up to coherent isomorphism and in the third up to coherent equivalence. The coherence conditions would then mirror those for tricategories \cite{GPS} with the possibility of some extra ones involving identities. However, we know of no genuine example of this kind of triple category except for those coming from tricategories, so there seems to be no advantage in developing such a theory at this time.

Our notion of intercategory is simpler and motivated by a number of examples of different kinds. As above there are three kinds of arrows and $2$-dimensional cells, and one kind of $3$-dimensional ones. The first type of arrow carries a category structure whereas the second {\em and} the third only satisfy the equations up to isomorphisms of the first type. In fact the first and second type of arrows with the corresponding cells form a weak double category, and so do the first and third. And here is the crucial point: the cells whose boundaries are arrows of the second and third kind have two compositions just like the other two but the interchange law doesn't hold. There is instead a comparison $3$-cell, in a specified direction, satisfying coherence conditions just like those for the duoidal categories of \cite{BS} (called $2$-monoidal categories in \cite{AM}).

Duoidal categories are one of our motivating examples. They are ``clearly'' a three-dimensional structure and, just as monoidal categories can be thought of as one-object bicategories, it is tempting to try to interpret duoidal categories as the one-object version of some more general kind of $3$-dimensional structure. And indeed duoidal categories are special kinds of intercategories. This gives a third way of viewing intercategories, as duoidal categories with several objects.

The plan for the paper is as follows. We introduce intercategories as a kind of lax triple category and set up our notation in Section \ref{intercategories}. We postpone the discussion of the coherence conditions until Sections \ref{3x3} and \ref{depiction}. After some double category preliminaries in Section \ref{prelim}, we give two more conceptual presentations of intercategory in Section \ref{pscat}. The equivalence of these is established in Section \ref{3x3} where a more symmetric notion is introduced. Here all the coherence conditions are given. In Section \ref{depiction} we reformulate these conditions using the more transparent $2$-dimensional notation of Section \ref{intercategories}.

In Section \ref{mor} we study morphisms of intercategories. There are three types, like for duoidal categories. There are those that are lax in both the horizontal and vertical directions, and those that are colax in them. There are also those that are colax in the horizontal direction and lax in the vertical. The other way around, horizontally lax and vertically colax, doesn't work because the arrows that should give the coherence diagrams do not compose. There are also double cells relating the three types of morphism in pairs. Finally, these double cells can be the faces of cubes whose commutativity is well-posed. The main theorem of the paper is Theorem \ref{mainth} which says that all of this fits together to produce a strict triple category of intercategories.

Somewhat mysteriously, horizontal intercategories only produce two of the three types of morphism and one kind of cell. Vertical intercategories give one of these and a new one, with their own cells. It is with double pseudocategories that we get all three kinds of morphism and three kinds of cell. The mystery is resolved in Section \ref{mystery} where we consider pseudocategories, not in a $2$-category, but in the {\em double category} $ \Doub $. With Theorem \ref{lastth} we can now rest assured that all presentations are equivalent.


\section{Intercategories}\label{intercategories}

An {\em intercategory} ${\ic A} $ is a $3$-dimensional structure whose constituents look like cubes
$$
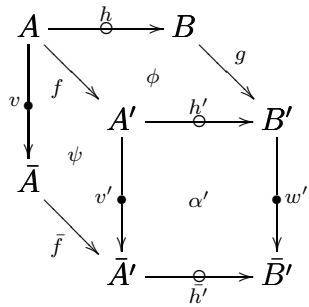
\bfig

\square(0,300)/@{>}|{\cc}`@{>}|{\bb}``/[A ` B `\ov{A} `;h ` v ` `]

\square(300,0)/@{>}|{\cc}`@{>}|{\bb}`@{>}|{\bb}`@{>}|{\cc}/[A' ` B' `\ov{A'}`\ov{B'}; h' ` v' `w' `\ov{h'}]

\morphism(0,800)|b|<300,-300>[A` A';f]

\morphism(500,800)<300,-300>[B ` B';g]

\morphism(0,300)|b|<300,-300>[\ov{A}`\ov{A'};\ov{f}]

\place(550,250)[\scriptstyle \alpha']

\place(150,400)[\scriptstyle \psi]

\place(400,650)[\scriptstyle \phi]

\efig
$$
It has:
\begin{itemize}

   \item {\em objects} ($A, B, \ldots$)

   \item three kinds of arrow; {\em transversal} ($f, g, \ldots$), {\em horizontal} ($h, h', \ldots$) and {\em vertical} ($v, v', \ldots$)

   \item three kinds of ($2$-dimensional) cell, {\em horizontal} ($\phi, \ldots$), {\em vertical} ($\psi, \ldots$), {\em basic} ($\alpha', \ldots$)

   \item $3$-dimensional cubes (not named in the diagram above).

\end{itemize}

\noindent It also has three kinds of composition:

\begin{itemize}

   \item {\em transversal} for transversal arrows, horizontal and vertical cells, and cubes. It is denoted by $ \cdot$ or juxtaposition, which is strictly associative and unitary. The identities are denoted $ 1 $

   \item {\em horizontal} for horizontal arrows, horizontal and basic cells, and cubes. It is denoted by $ \circ $ or $|$, and is associative and unitary up to coherent transversal isomorphism

   \item {\em vertical} for vertical arrows, vertical and basic cells, and cubes. It is denoted by $\bullet $ or $-$, and is associative and unitary up to coherent transversal isomorphism

\end{itemize}

Mnemonic: As we go from transversal to horizontal to vertical, the notation becomes heavier (uses more ink).

These structures interact as follows:

\begin{itemize}

   \item transversal and horizontal form a weak double category

   \item transversal and vertical form a weak double category

   \item horizontal and vertical are related by {\em interchangers}

$$
 \chi : (\alpha \circ \beta) \bullet (\ov{\alpha}\circ \ov{\beta}) \to (\alpha \bullet \ov{\alpha}) \circ (\beta \bullet \ov{\beta}) 
$$
which is also written as
$$
\chi : \frac{\alpha | \beta}{\ov{\alpha} | \ov{\beta}}     \to \left.\dfrac{\alpha}{\ov{\alpha}}\right|\dfrac{\beta}{\ov{\beta}}
$$
An important feature of this notation is that the $ \alpha, \beta, \ov{\alpha}, \ov{\beta} $ don't change place.

The other interchangers are
$$
\mu : \id_v \bullet \id_{\ov{v}} \to \id_{v\bullet \ov{v}}, \quad  \delta : \Id_{h\circ h'} \to \Id_h \circ \Id_{h'}, \quad  \tau : \Id_{\id_A} \to \id_{\Id_A}
$$
also written as
$$
\mu :  \frac{\id_v}{\id_{\ov{v}}} \to \id_{\frac{v}{\ov{v}}}, \quad  \delta : \Id_{h | h'} \to \Id_h | \Id_{h'}, \quad  \tau : \Id_{\id_A} \to \id_{\Id_A}
$$
These are transversally special cubes, i.e.~their horizontal and vertical faces are transversal identities. They compose in the transversal direction.

For example, $ \mu $ is the cube (shown as a morphism from back to the front)
$$
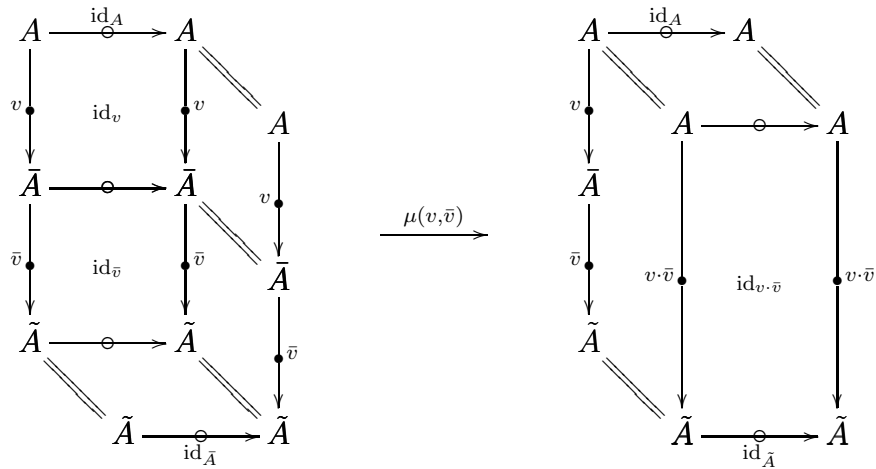
\bfig

\square(0,300)/@{>}|{\cc}`@{>}|{\bb}`@{>}|{\bb}`@{>}|{\cc}/[\ov{A} `\ov{A} `\ovv{A} `\ovv{A};`\ov{v}`\ov{v}`]

\square(0,800)/@{>}|{\cc}`@{>}|{\bb}`@{>}|{\bb}`@{>}|{\cc}/[A `A`\ov{A} `\ov{A};\id_A `v`v`]

\square(300,0)/``@{>}|{\bb}`@{>}|{\cc}/[`\ov{A} `\ovv{A}`\ovv{A};``\ov{v}`\id_{\ov{A}}]

\morphism(800,1000)/@{>}|{\bb}/<0,-500>[A `\ov{A};v]

\morphism(0,300)/=/<300,-300>[\ovv{A}`\ovv{A};]

\morphism(500,300)/=/<300,-300>[\ovv{A}`\ovv{A};]

\morphism(500,800)/=/<300,-300>[\ov{A}`\ov{A};]

\morphism(500,1300)/=/<300,-300>[A`A;]

\place(250,550)[\scriptstyle \id_{\ov{v}}]

\place(250,1050)[\scriptstyle \id_v]

\morphism(610,950)/=/<50,-50>[`;]

\morphism(610,450)/=/<50,-50>[`;]

\morphism(350,175)/=/<50,-50>[`;]

\place(1300,650)[\to^{\mu(v,\ov{v})}]

\square(1800,800)/@{>}|{\cc}`@{>}|{\bb}``/[A `A`\ov{A}`;\id_A`v``]

\morphism(1800,800)|l|/@{>}|{\bb}/<0,-500>[\ov{A}`\ovv{A};\ov{v}]

\square(2100,0)/@{>}|{\cc}`@{>}|{\bb}`@{>}|{\bb}`@{>}|{\cc}/<500,1000>[A`A`\ovv{A}`\ovv{A};`v\cdot\ov{v}`v\cdot\ov{v}`\id_{\ovv{A}}]

\morphism(1800,300)/=/<300,-300>[\ovv{A}`\ovv{A};]

\morphism(1800,1300)/=/<300,-300>[A`A;]

\morphism(2300,1300)/=/<300,-300>[A`A;]

\place(2350,500)[\scriptstyle \id_{v\cdot\ov{v}}]

\morphism(1900,700)/=/<50,-50>[`;]

\morphism(2150,1200)/=/<50,-50>[`;]

\efig
$$
Naturality is in the transversal direction. For example, $\id $ can also be applied to vertical cells $ \psi $ giving a cube $ \id_\psi $

$$
\bfig

\square(0,300)/@{>}|{\cc}`@{>}|{\bb}``/[A ` A `\ov{A} `;\id_A `v ` `]

\square(300,0)/@{>}|{\cc}`@{>}|{\bb}`@{>}|{\bb}`@{>}|{\cc}/[A' ` A' `\ov{A'}`\ov{A'};  ` v' `v' `\id_{\ov{A'}}]

\morphism(0,800)|b|<300,-300>[A` A';f]

\morphism(500,800)<300,-300>[A ` A';f]

\morphism(0,300)|b|<300,-300>[\ov{A}`\ov{A'};\ov{f}]

\place(550,250)[\scriptstyle \id_{v'}]

\place(150,400)[\scriptstyle \psi]

\place(400,650)[\scriptstyle \id_f]

\efig
$$
needed to express naturality of $ \mu $

$$
\bfig

\square[\dfrac{\id_v}{\id_{\ov{v}}}   `\id_\frac{v}{\ov{v}} `\dfrac{\id_{v'}}{\id_{\ov{v'}}}`\id_\frac{v'}{\ov{v'}};
\mu(v, \ov{v}) `\frac{\id_\psi}{\id_{\ov{\psi}}}
` \id_\frac{\psi}{\ov{\psi}}  
`\mu(v', \ov{v'})]

\efig
$$

\end{itemize}

\noindent These have to satisfy all the obvious coherence conditions, of which there are many (32 by one way of counting), although this is below average for weak $3$-dimensional structures. We return to this in Section \ref{3x3}. 

It is thus desirable to have a more abstract theoretical presentation. This will be especially useful in studying the various kinds of morphism of intercategories. The different points of view provide extra insight into the concepts, and can prove useful in checking the various examples. These different ways of viewing intercategories are presented in the following sections (\ref{pscat} and \ref{3x3}), where the coherence conditions will be made explicit. 

We give two examples to guide the reader through the maze of $3$-dimensional commutative diagrams. Many more may be found in \cite{PartII}.

The first is given by duoidal categories, categories $ {\bf D} $ with two tensors $ \otimes $ and $ \boxtimes $ related by lax interchange (see \cite{AM, BCZ, BS}). (We follow most closely the notation of \cite{BS} although they use $ \gamma$ for the interchanger which we call $ \chi $.) This can be considered as a degenerate intercategory with one object. The objects of $ {\bf D} $ provide the basic cells, and the morphisms of ${\bf D} $ the cubes. The remaining constituents are all identities. A general cube looks like
$$
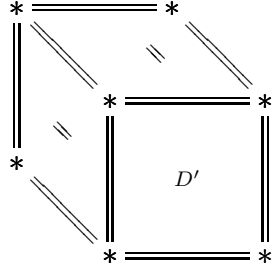
\bfig

\square(0,300)/=`=``/[*`*`*`;` ` `]

\square(300,0)/=`=`=`=/[*`*`*`*; ```]

\morphism(0,800)|b|/=/<300,-300>[*`*;]

\morphism(500,800)/=/<300,-300>[* ` *;]

\morphism(0,300)|b|/=/<300,-300>[*`*;]

\place(550,250)[\scriptstyle D']

\morphism(100,450)/=/<70,-70>[`;]

\morphism(400,700)/=/<70,-70>[`;]

\efig
$$ 
with an arrow $ d : D \to D' $ inside. Horizontal and vertical composition are given by $ \otimes $ and $ \boxtimes $ respectively. In this example all of the interchangers $ \chi $, $ \delta $, $ \mu $, $ \tau$ can be non-invertible, for example when the horizontal composition is product and the vertical coproduct.

The second example is given by spans of cospans in a category with pullbacks and pushouts. A general cube looks like
$$
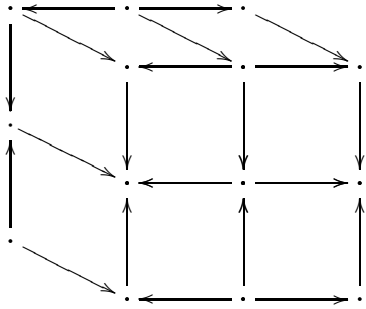
\bfig\scalefactor{.75}

\square/<-`<-`<-`<-/[\cdot `\cdot `\cdot `\cdot;```]

\square(500,0)/>`<-`<-`>/[\cdot`\cdot`\cdot`\cdot;```]

\square(0,500)/<-`>`>`<-/[\cdot `\cdot `\cdot `\cdot;```]

\square(500,500)/>`>`>`>/[\cdot `\cdot `\cdot `\cdot;```]

\morphism(0,1250)/>/<-500,0>[\cdot`\cdot;]

\morphism(-500,1250)/>/<500,-250>[\cdot`\cdot;]

\morphism(0,1250)/>/<500,0>[\cdot`\cdot;]

\morphism(0,1250)/>/<500,-250>[\cdot`\cdot;]

\morphism(500,1250)/>/<500,-250>[\cdot`\cdot;]

\morphism(-500,1250)/>/<0,-500>[\cdot`\cdot;]

\morphism(-500,750)/>/<500,-250>[`\cdot;]

\morphism(-500,750)/<-/<0,-500>[`\cdot;]

\morphism(-500,250)/>/<500,-250>[\cdot`\cdot;]

\efig
$$
Horizontal (resp. vertical) composition is given by pullback (resp. pushout). Note that because the interchanger $ \chi $ is not invertible, the spans must be horizontal and the cospans vertical. More details may be found in \cite{PartII}.


\section{Preliminaries on double categories}\label{prelim}

Recall from \cite{GP} the {\em strict} double category $ \Doub $ whose objects are weak double categories, horizontal arrows are lax functors (of double categories), and vertical arrows are colax functors. It is, of course, the double cells that are of most interest, given the well-known problems with lax transformations in general. A double cell
$$
\bfig
\square/>`@{>}|{\bb}`@{>}|{\bb}`>/[{\mathbb A} ` {\mathbb B} `{\mathbb C} `{\mathbb D};F ` U ` V `G]

\place(250,250)[\scriptstyle \pi]
\efig
$$
is given by

\begin{itemize}

\item [(1)] a horizontal arrow $ \pi A : VFA \to GUA $ for each object $ A $ of $ {\mathbb A} $, and

\item [(2)] a double cell
$$
\bfig
\square/>`@{>}|{\bb}`@{>}|{\bb}`>/[VFA ` GUA `VF \ov{A} ` GU\ov{A};\pi A ` VFu ` GUu `\pi\ov{A}]

\place(250,250)[\scriptstyle \pi u]
\efig
$$
for each vertical arrow $ u : A \tod \ov{A} $ of $ {\mathbb A} $.

\end{itemize}

\noindent These are required to satisfy a number of more or less obvious compatibility conditions which can be found in \cite{GP}. Note, in passing, that in writing out these conditions one realizes that $ \pi A $ has to be in the direction given.

From this double category $ \Doub $ we can extract two $2$-categories $ \Dlax $ and $ \Dcolax $ by restricting the vertical (resp.\ horizontal) arrows to be identities. So the objects of $ \Dlax $ are weak double categories, the arrows are lax functors, and the $2$-cells are horizontal transformations. $ \Dcolax $ has the same objects, the colax functors as arrows and again horizontal transformations as $2$-cells.

\noindent {\sc Remark}: It's not unreasonable to wonder where the usual problems of making lax functors into a $2$-category have gone, given that $2$-categories and bicategories can be considered as special weak double categories. In the case of $2$-categories considered as horizontal double categories, lax functors are just $2$-functors and horizontal transformations are $2$-natural transformations so there are no problems there. If, on the other hand, bicategories are considered as vertical weak double categories (i.e.\ horizontal arrows are identities), then lax functors are what are usually called lax functors but horizontal transformations are now the (dual) icons of Lack \cite{L}, whose main feature is precisely that they are the $2$-cells of a $2$-category.

Recall that a lax functor $ F : {\mathbb A} \to {\mathbb B} $ takes objects, horizontal and vertical arrows and double cells of $ {\mathbb A} $ to like ones in $ {\mathbb B} $ as illustrated in
$$
\bfig
\square/>`@{>}|{\bb}`@{>}|{\bb}`>/[A ` A'`\ov{A}`\ov{A'};a ` v` v'`\ov{a}]
\place(250,250)[\scriptstyle \alpha]
\place(800,250)[\longmapsto]
\place(800,300)[\scriptstyle F]
\square(1200,0)/>`@{>}|{\bb}`@{>}|{\bb}`>/[FA`FA'`F\ov{A}`F\ov{A'};Fa `Fv`Fv'`F\ov{a}]
\place(1450,250)[\scriptstyle F\alpha]
\efig
$$
$ F $ is required to be functorial in the horizontal direction, whereas in the vertical direction comparison special cells (horizontal domains and codomains identities) are provided
$$
\bfig
\square(0,250)/=`@{>}|{\bb}`@{>}|{\bb}`=/[FA `FA`FA`FA;`\id_{FA}`F(\id_A)`]
\place(250,500)[\scriptstyle \phi A]
\square(1000,0)/`@{>}|{\bb}``=/[F\ov{A}``F\ovv{A}`F\ovv{A};`F\ov{v}``]
\square(1000,500)/=`@{>}|{\bb}``/[FA`FA`F\ov{A}`;`Fv``]
\square(1000,0)/=``@{>}|{\bb}`=/<500,1000>[FA` FA`F\ovv{A}`F\ovv{A};``F(v \cdot \ov{v})`]
\place(1300,500)[\scriptstyle \phi(v,\ov{v})]
\efig
$$
$$
\phi A : \id_{FA} \to F(\id_A) \ \ \ \ \   \   \phi(v,\ov{v}) : F v \cdot F\ov{v} \to F (v \cdot \ov{v})
$$
which are required to satisfy
$$
\bfig
\square/>`>`>`<-/<800,500>[Fv \cdot \id_{F\ov{A}} `Fv\cdot F(\id_{\ov{A}}) `Fv` F(v\cdot \id_{\ov{A}});Fv\cdot \phi \ov{A} ` \rho ` \phi(v,\id_{\ov{A}}) `F\rho]
\square(1500,0)/>`>`>`<-/<800,500>[\id_{F A} \cdot Fv ` F(\id_A)\cdot Fv  `Fv` F(\id_A\cdot v);
        \phi A \cdot Fv `\lambda ` \phi(\id_A,v)`F\lambda]
\efig
$$
$$
\bfig
\square/>`>``>/<1000,500>[F v \cdot (F\ov{v}\cdot F\ovv{v})
       ` F v\cdot F(\ov{v}\cdot \ovv{v})
        `(F v\cdot F \ov{v}) \cdot F\ovv{v}
        `F(v\cdot \ov{v}) \cdot F\ovv{v};
                      Fv\cdot \phi(\ov{v},\ovv{v})
                      `\kappa `
                       `\phi(v,\ov{v})\cdot F\ovv{v}]
\square(1000,0)/>``>`>/<1000,500>[Fv\cdot F(\ov{v}\cdot \ovv{v})
           `F(v\cdot(\ov{v}\cdot \ovv{v}))
            `F(v\cdot \ov{v}) \cdot F\ovv{v}
             `F((v\cdot\ov{v})\cdot\ovv{v});
              \phi(v, \ov{v} \cdot \ovv{v})
              ``F\kappa
              `\phi(v\cdot\ov{v}, \ovv{v})]
\efig
$$
If the $ \phi$'s are isomorphisms, $ F $ is said to be {\em strong}, and if they are identities, $ F $ is {\em strict}. So strict functors preserve, not only vertical identities and composition on the nose, but also the structural isomorphisms $ \rho $, $ \lambda $, $\kappa $, as can be seen from the above diagrams.

\ 

\noindent {\sc Remark:} Strict functors go counter to the received wisdom about morphisms of (weak) two dimensional structures but they arise surprisingly often, e.g.\ the projections out of comma objects are strict, and are very useful. The following result, whose proof is straightforward, illustrates this and is crucial for us.

\begin{proposition} Let $ F : {\mathbb A} \to {\mathbb C} $ and $ G : {\mathbb B} \to {\mathbb C} $ be strict functors. Then the set-theoretical pullback of $ F $ and $ G $ on objects, arrows and cells defines a weak double category $ {\mathbb A} \times_{\mathbb C} {\mathbb B}$ which is the $2$-pullback of $ F $ and $ G $ in $\Dlax $. The projections onto $ {\mathbb A} $ and $ {\mathbb B} $ are strict functors.

\end{proposition}

\noindent {\sc Remark}: The $2$-pullback of $ F $ and $ G $ may exist even when $ F $ and $ G $ are not strict. For example, the images of $ F $ and $ G $ may be disjoint, in which case the empty double category is the $2$-pullback. It is not clear exactly which pullbacks exist in $ \Dlax $, nor what the $2$-categorical property of strict functors is that makes their pullbacks better.


\section{Horizontal and vertical intercategories}\label{pscat}

The notion of pseudocategory in a $2$-category with $2$-pullbacks is pretty clear and has already appeared in print \cite{PC}. We generalize this a bit to suit our purposes: we don't assume that $2$-pullbacks exist but restrict our definition to diagrams where the requisite ones do.

Given a diagram
$$
B \two^{\partial_0}_{\partial_1} A
$$
in a $2$-category ${\cal A} $, the $2$-limit of the diagram
$$
\bfig\scalefactor{.7}
\Vtriangle/`>`>/<350,300>[B`B`A;`\partial_1`\partial_0]
\Atriangle(350,0)/>`>`/<350,300>[B`A`A;``]
\Vtriangle(700,0)/`>`>/<350,300>[B``A;`\partial_1`\partial_0]
\place(1500,300)[\cdots]
\Vtriangle(1600,0)/`>`>/<350,300>[`B`A;`\partial_1`\partial_0]
\efig
$$
can be constructed using $2$-pullbacks. We call this limit the {\em iterated $2$-pullback} and denote it by $ B\times_A B\times_A \dots \times_A B = B^{(n)} $ and the projections by
$$
p_i : B^{(n)} \to B,\  i = 1, \dots, n.
$$
Also let $ B^{(0)} = A $ and $ B^{(1)} = B $. If $ f : B^{(n)} \to B^{(m)} $ we will denote $ B^{(r)} \times_A f\times_A B^{(s)} $ by
$$
f_{r+1,\dots, r+n} : B^{(r+n+s)} \to B^{(r+m+s)}.
$$

\begin{definition}\label{icdef} A {\em pseudocategory object} in $ {\cal A} $ is a diagram $ B\two^{\partial_0}_{\partial_1} A $ for which the iterated $2$-pullbacks $ B^{(n)} $ exist, together with morphisms
   \begin{itemize}
           \item[(1)] $\id : A \to B $

            \item[(2)] $ m : B^{(2)} \to B $

and iso $2$-cells

            \item[(3)]
$
\bfig
\place(100,0)[\ \ \ ]
\square(500,0)[B^{(3)} ` B^{(2)} ` B^{(2)} ` B; m_{23} `m_{12} `m`m]
\place(750,250)[{\scriptstyle \kappa} \Downarrow]
\efig
$

           \item[\ \ \ \ \ \ (4)/(5)]
$
\bfig
\place(100,0)[\ \ \ ]
\qtriangle(500,0)/>`>`>/<650,500>[B`B^{(2)}`B;\id_1`=`m]
\place(1000,350)[{\scriptstyle \lambda} \Downarrow]
\ptriangle(1150,0)/<-`>`>/<650,500>[B^{(2)} `B`B;\id_2``=]
\place(1400,350)[{\scriptstyle \rho} \Downarrow]

\efig
$

\noindent satisfying the usual coherence conditions (pentagon, etc.).

   \end{itemize}

Note that a pseudocategory in $ \Cat $ is exactly a weak double category.

\end{definition}

\begin{definition} \label{hintercat} A {\em horizontal intercategory} is a pseudocategory
$$
{\mathbb C} \threepppp/>`>`>/<400>^{p_1} |{m} _{p_2} {\mathbb B} \three/>`<-`>/^{\partial_0} | {\id}_{\partial_1} {\mathbb A}
$$
in $\Dlax $ where $ \partial_0 $ and $ \partial_1 $ are strict functors.

\end{definition}

In this definition, it is understood that
$$
\bfig
\square[{\mathbb C} ` {\mathbb B} `{\mathbb B} `{\mathbb A};p_1 `p_2 ` \partial_1 `\partial_0]
\efig
$$
is a pullback so that $ p_1 $ and $ p_2 $ are also strict and thus the pullbacks needed to express coherence also exist.

In this presentation the horizontal and vertical arrows of $ {\mathbb A} $ are the transversal and vertical arrows of the $ {\ic A} $ of Section \ref{intercategories}, so the cells of $ {\mathbb A} $ are the vertical cells of $ {\ic A} $. The objects, horizontal and vertical arrows, and cells of $ {\mathbb B} $ are the horizontal arrows, horizontal and basic cells, and cubes of $ {\ic A} $. (See table at the end of Section \ref{3x3}.)

Given two pseudocategories $ C \threepppp/>`>`>/<400>^{p_1} |{m}_{p_2} B \three/>`<-`>/^{\partial_0} | {\id} _{\partial_1} A $ and $ C' \threepppp/>`>`>/<400>^{p'_1} |{m'}_{p'_2} B' \three/>`<-`>/^{\partial'_0} |{\id'} _{\partial'_1} A' $ in a $2$-category, a {\em lax functor} from the first to the second consists of three arrows $ f $, $ g $, $ h$ such that
$$
\bfig
\square/`>`>`/[C `B`C'`B';`h`g`]
\place(240,500)[\two^{p_1}_{p_2}]
\place(240,0)[\two^{p'_1}_{p'_2}]
\square(500,0)/`>`>`/[B`A`B'`A';``f`]
\place(740,500)[\two^{\partial_0}_{\partial_1}]
\place(740,0)[\two^{\partial'_0}_{\partial'_1}]
\efig
$$
``commutes sequentially'' i.e.\ $ f\partial_i = \partial'_i g$, and $ gp_j = p'_j h$ for $ i = 0, 1 $ and $ j = 1, 2 $. (So $ h $ is in fact determined by $ f $ and $g $.) Apart from these three arrows there are also given comparison $2$-cells
$$
\bfig
\square[A ` B ` A' ` B';\id ` f ` g `\id']
\place(250,240)[\Rightarrow]
\place(250,320)[\scriptstyle \eta]
\square(800,0)[C` B ` C' ` B';m `h`g`m']
\place(1050,240)[\Rightarrow]
\place(1050,320)[\scriptstyle \mu]
\efig
$$
satisfying the usual equations, i.e.\ unit and associativity.

There are also colax morphisms of pseudocategory objects with the same conditions except that $ \eta $ and $ \mu $ go in the opposite direction.

Returning now to intercategories as pseudocategories in $\Dlax $, we get two kinds of morphisms, lax and colax
$$
\bfig
\square/`>`>`/[{\mathbb C} `{\mathbb B}`{\mathbb C'}`{\mathbb B'};`H`G`]
\place(240,500)[\threepppp/>`>`>/<300>^{}|{}_{}]
\place(240,0)[\threepppp/>`>`>/<300>^{}|{}_{}]
\square(500,0)/`>`>`/[{\mathbb B} `{\mathbb A} `{\mathbb B'}`{\mathbb A'};``F`]
\place(740,500)[\threepppp/>`<-`>/<300>^{}|{}_{}]
\place(740,0)[\threepppp/>`<-`>/<300>^{}|{}_{}]
\efig
$$
but the $ F $, $ G $, $ H $ are always lax. We analyze them further in the next section. 

There is a different way of looking at intercategories as pseudocategories, i.e. in the $2$-category $ \Dcolax $. We get an equivalent notion of intercategory but not the same notions of morphisms. We'll call them vertical intercategories. In the next section, we prove that horizontal and vertical intercategories are equivalent notions.

\begin{definition}\label{vintercat} A {\em vertical intercategory} is a pseudocategory
$$
{\mathbb X}_2 \threepppp/>`>`>/<400>^{P_1} |{M} _{P_2} {\mathbb X}_1 \three/>`<-`>/^{D_0} | {\Id}_{D_1} {\mathbb X}_0
$$
in $\Dcolax $ where $ D_0 $ and $ D_1 $ are strict functors.

\end{definition}

Now, the objects, horizontal and vertical arrows, and cells of $ {\mathbb X}_0 $ are the objects, transversal and horizontal arrows, and horizontal cells of $ {\ic A}$, respectively. $ {\mathbb X}_1 $ consists of vertical arrows, vertical and basic cells, and cubes of $ {\ic A} $. (See table at the end of Section \ref{3x3}.)


\section{Double pseudocategories in a $2$-category}\label{3x3}

With a view to understanding the equivalence of the two preceding definitions of intercategory, we present a more symmetric version which doesn't refer to lax or colax. Although we only need it for $\CAT $, the definition we give below works in any $2$-category with $2$-pullbacks, with a simple change of word here and there. This may be useful, say for ${\bf V}$-$\CAT$ for appropriate monoidal categories $ {\bf V} $, but just the fact that the notion can be thus internalized says something about the definition.

If we make the double categories $ {\mathbb A} $, ${\mathbb B} $, ${\mathbb C} $ of Definition \ref{hintercat} explicit as pseudocategories in $ \CAT $ we get a $ 3 \times 3 $ diagram of categories and functors
$$
\bfig

\square/>`<-`<-`>/[{\bf C}_1 `{\bf B}_1`{\bf C}_2`{\bf B}_2;```]

\square|allb|/@{>}@<-3pt>`@{<-}@<-3pt>`@{<-}@<-3pt>`@{>}@<-3pt>/[{\bf C}_1 `{\bf B}_1`{\bf C}_2`{\bf B}_2;```]

\square|allb|/@{>}@<3pt>`@{<-}@<3pt>`@{<-}@<3pt>`@{>}@<3pt>/[{\bf C}_1 `{\bf B}_1`{\bf C}_2`{\bf B}_2;```]

\square(0,500)/>`>`>`>/[{\bf C}_0 `{\bf B}_0`{\bf C}_1`{\bf B}_1;```]

\square(0,500)|allb|/@{>}@<-3pt>`@{<-}@<-3pt>`@{<-}@<-3pt>`@{>}@<-3pt>/[{\bf C}_0 `{\bf B}_0`{\bf C}_1`{\bf B}_1;```]

\square(0,500)|allb|/@{>}@<3pt>`@{<-}@<3pt>`@{<-}@<3pt>`@{>}@<3pt>/[{\bf C}_0 `{\bf B}_0`{\bf C}_1`{\bf B}_1;```]

\square(500,0)/<-`<-`<-`<-/[{\bf B}_1 `{\bf A}_1`{\bf B}_2`{\bf A}_2;```]

\square(500,0)|allb|/@{>}@<-3pt>`@{<-}@<-3pt>`@{<-}@<-3pt>`@{>}@<-3pt>/[{\bf B}_1 `{\bf A}_1`{\bf B}_2`{\bf A}_2;```]

\square(500,0)|allb|/@{>}@<3pt>`@{<-}@<3pt>`@{<-}@<3pt>`@{>}@<3pt>/[{\bf B}_1 `{\bf A}_1`{\bf B}_2`{\bf A}_2;```]

\square(500,500)/<-`>`>`<-/[{\bf B}_0 `{\bf A}_0`{\bf B}_1`{\bf A}_1;```]

\square(500,500)|allb|/@{>}@<-3pt>`@{<-}@<-3pt>`@{<-}@<-3pt>`@{>}@<-3pt>/[{\bf B}_0 `{\bf A}_0`{\bf B}_1`{\bf A}_1;```]

\square(500,500)|allb|/@{>}@<3pt>`@{<-}@<3pt>`@{<-}@<3pt>`@{>}@<3pt>/[{\bf B}_0 `{\bf A}_0`{\bf B}_1`{\bf A}_1;```]

\place(-1000,500)[\ \ ]
\place(2000,500)[(*)]

\efig
$$

A horizontal intercategory is a bit less than this and a bit more too. ``A bit less'' because some of those categories and functors are determined by the others. For example, we only need to specify the categories $ {\bf A}_0, {\bf B}_0, {\bf A}_1, {\bf B}_1 $, but we need the others to serve as domains for the composition functors.  ``A bit more'' because we need an extra row and column in order to specify the structural isomorphisms (associativity, etc.) and one more row and column to express the coherence conditions.

The names of the 36 functors follow those of Definition \ref{icdef}. The vertical ones are the uppercase versions of the corresponding horizontal ones. The first condition on $ (*) $ is that each of the rows $ {\mathbb X}_0, {\mathbb X}_1, {\mathbb X}_2$ and each of the columns $ {\mathbb A}, {\mathbb B}, {\mathbb C} $ represent weak double categories. In particular the following diagrams are pullbacks in ${\cal CAT} $:
$$
\bfig
\square[{\bf C}_i  ` {\bf B}_i ` {\bf B}_i ` {\bf A}_i; p^i_2 `p^i_1 `\partial^i_0 `\partial^i_1]
\place(1200,250)[i = 0, 1, 2]

\place(250,250)[\scriptstyle (1)]

\place(-1000,0)[\ ]
\efig
$$
and
$$
\bfig
\square[{\bf A}_2  `{\bf A}_1  `{\bf A}_1  `{\bf A}_0;P^A_2  `P^A_1  `D^A_0  `D^A_1]
\place(250,250)[\scriptstyle (2)]

\square(1000,0)[{\bf B}_2  ` {\bf B}_1  `{\bf B}_1  `{\bf B}_0;P^B_2 `P^B_1  `D^B_0 `D^B_1]
\place(1250,250)[\scriptstyle (3)]

\square(2000,0)[{\bf C}_2  `{\bf C}_1  `{\bf C}_1  `{\bf C}_0;P^C_2  `P^C_1  `D^C_0  `D^C_1]
\place(2250,250)[\scriptstyle (4)]
\efig
$$
Each of these weak double categories comes with its structural natural isomorphisms, usually clear from the context, denoted $ \kappa, \lambda, \rho $ in the first case and $\kappa', \lambda', \rho'$ in the second.

The domains and codomains as well as the pullback projections preserve all the structure. More precisely, the following diagrams commute.

\noindent Those involving $ \partial_0, \partial_1 $
$$
\bfig
\square/>`<-`<-`>/[{\bf B}_0 ` {\bf A}_0 `{\bf B}_1 `{\bf A}_1; \partial^0_i  `D^B_j `D^A_j `\partial^1_i]
\place(250,250)[\scriptstyle (5)]
\square(1000,0)/>`<-`<-`>/[{\bf B}_1 `{\bf A}_1  `{\bf B}_2 `{\bf A}_2; \partial^1_i  `P^B_k `P^A_k  `\partial^2_i]
\place(1250,250)[\scriptstyle (6)]
\square(2000,0)[{\bf B}_0  `{\bf A}_0  `{\bf B}_1 `{\bf A}_1;\partial^0_i  `\Id^B  `\Id^A `\partial^1_i]
\place(2250,250)[\scriptstyle (7)]
\square(3000,0)/>`<-`<-`>/[{\bf B}_1 `{\bf A}_1  `{\bf B}_2  `{\bf A}_2;\partial^1_i `M^B `M^A `\partial^2_i]
\place(3250,250)[\scriptstyle (8)]

\efig
$$

\noindent Those involving $ p_1, p_2 $
$$
\bfig
\square/>`<-`<-`>/[{\bf C}_0 ` {\bf B}_0 `{\bf C}_1 `{\bf B}_1; p^0_k  `D^C_i ` D^B_i `p^1_k]
\place(250,250)[\scriptstyle (9)]
\square(1000,0)/>`<-`<-`>/[{\bf C}_1 `{\bf B}_1  `{\bf C}_2 `{\bf B}_2; p^1_k `P^C_i `P^B_l `p^2_k]
\place(1250,250)[\scriptstyle (10)]
\square(2000,0)[{\bf C}_0  `{\bf B}_0  `{\bf C}_1 `{\bf B}_1;p^0_k `\Id^C `\Id^B `p^1_k]
\place(2250,250)[\scriptstyle (11)]
\square(3000,0)/>`<-`<-`>/[{\bf C}_1 `{\bf B}_1  `{\bf C}_2  `{\bf B}_2;p^1_k `M^C `M^B `p^2_k]
\place(3250,250)[\scriptstyle (12)]

\efig
$$

\noindent Those involving $ \id, m$
$$
\bfig
\square/<-`<-`<-`<-/[{\bf B}_0 `{\bf A}_0 `{\bf B}_1 `{\bf A}_1;\id^0 `D^B_i `D^A_i `\id^1]
\place(250,250)[\scriptstyle (13)]
\square(1000,0)/<-`<-`<-`<-/[{\bf B}_1 `{\bf A}_1 `{\bf B}_2 `{\bf A}_2;\id^1 `P^B_k `P^A_k `\id^2]
\place(1250,250)[\scriptstyle (14)]
\square(2000,0)/>`<-`<-`>/[{\bf C}_0  `{\bf B}_0  `{\bf C}_1 `{\bf B}_1;m^0 `D^C_i `D^B_i`m^1]
\place(2250,250)[\scriptstyle (15)]
\square(3000,0)/>`<-`<-`>/[{\bf C}_1 `{\bf B}_1  `{\bf C}_2  `{\bf B}_2;m^1 `P^C_k `P^B_k `m^2]
\place(3250,250)[\scriptstyle (16)]

\efig
$$

\noindent The crucial ingredients are the natural transformations
$$
\bfig
\square/<-`>`>`<-/[{\bf B}_0 `{\bf A}_0 `{\bf B}_1 `{\bf A}_1;\id^0 `\Id^B `\Id^A `\id^1]
\place(250,-200)[(17)]
\place(230,250)[\twoar(1,-1)]
\place(290,250)[\scriptstyle \tau]

\square(1000,0)/<-`<-`<-`<-/[{\bf B}_1 `{\bf A}_1 `{\bf B}_2 `{\bf A}_2;\id^1 `M^B `M^A `\id^2]
\place(1250,-200)[(18)]
\place(1270,210)[\twoar(1,1)]
\place(1210,270)[\scriptstyle \mu]

\square(2000,0)[{\bf C}_0  `{\bf B}_0  `{\bf C}_1 `{\bf B}_1;m^0 `\Id^C `\Id^B`m^1]
\place(2250,-200)[(19)]
\place(2290,230)[\twoar(-1,-1)]
\place(2270,240)[\scriptstyle \delta]

\square(3000,0)/>`<-`<-`>/[{\bf C}_1 `{\bf B}_1  `{\bf C}_2  `{\bf B}_2;m^1 `M^C `M^B_k `m^2]
\place(3250,-200)[(20)]
\place(3200,180)[\twoar(-1,1)]
\place(3280,270)[\scriptstyle \chi]

\efig
$$

Apart from the coherence conditions on $ \kappa, \lambda, \rho, \kappa', \lambda', \rho' $, giving pseudocategories, the following must also hold.

\noindent Those involving $\id $
$$
\bfig\scalefactor{.8}
\Vtriangle/`>`<-/<500,300>[{\bf B}_3 `{\bf B}_1`{\bf B}_2;`M_{23}`M]
\Vtriangle(0,700)/`>`<-/<500,300>[{\bf A}_3`{\bf A}_1`{\bf A}_2;`M_{23}`M]
\Atriangle(0,1000)/<-`>`/<500,300>[{\bf A}_2`{\bf A}_3`{\bf A}_1;M_{12}`M`]
\morphism(0,1000)|l|<0,-700>[{\bf A}_3`{\bf B}_3;\id]
\morphism(500,700)|l|<0,-700>[{\bf A}_2`{\bf B}_2;\id]
\morphism(1000,1000)|r|<0,-700>[{\bf A}_1`{\bf B}_1;\id]

\place(1500,650)[=]

\Vtriangle(2000,0)/`>`<-/<500,300>[{\bf B}_3 `{\bf B}_1`{\bf B}_2;`M_{23}`M]
\Atriangle(2000,300)/<-`>`/<500,300>[{\bf B}_2`{\bf B}_3`{\bf B}_1;M_{12}`M`]
\Atriangle(2000,1000)/<-`>`/<500,300>[{\bf A}_2`{\bf A}_3`{\bf A}_1;M_{12}`M`]
\morphism(2000,1000)|l|<0,-700>[{\bf A}_3`{\bf B}_3;\id]
\morphism(2500,1300)|l|<0,-700>[{\bf A}_2`{\bf B}_2;\id]
\morphism(3000,1000)|r|<0,-700>[{\bf A}_1`{\bf B}_1;\id]

\morphism(250,400)|l|/=>/<0,200>[`;\mu_{23}]
\morphism(750,400)|l|/=>/<0,200>[`;\mu]
\morphism(500,950)|l|/=>/<0,200>[`;\kappa']

\morphism(2250,700)|l|/=>/<0,200>[`;\mu_{12}]
\morphism(2750,700)|l|/=>/<0,200>[`;\mu]
\morphism(2500,200)|l|/=>/<0,200>[`;\kappa']

\place(-700,650)[\ \ ]
\place(3700,650)[(21)]

\efig
$$

$$
\bfig\scalefactor{.8}
\Vtriangle/`>`<-/<500,300>[{\bf B}_1 `{\bf B}_1`{\bf B}_2;`\Id_1`M]
\Vtriangle(0,700)/=`>`<-/<500,300>[{\bf A}_1`{\bf A}_1`{\bf A}_2;`\Id_1`M]

\morphism(0,1000)|l|<0,-700>[{\bf A}_1`{\bf B}_1;\id]
\morphism(500,700)|l|<0,-700>[{\bf A}_2`{\bf B}_2;\id]
\morphism(1000,1000)|r|<0,-700>[{\bf A}_1`{\bf B}_1;\id]

\morphism(250,400)|l|/=>/<0,200>[`;\tau_1]
\morphism(750,400)|l|/=>/<0,200>[`;\mu]
\morphism(500,800)|l|/=>/<0,170>[`;\lambda']

\place(1500,650)[=]

\Vtriangle(2000,0)/=`>`<-/<500,300>[{\bf B}_1 `{\bf B}_1`{\bf B}_2;`\Id_1`M]
\square(2000,300)/=`>`>`=/<1000,700>[{\bf A}_1`{\bf A}_1`{\bf B}_1`{\bf B}_1;`\id`\id`]

\morphism(2500,100)|l|/=>/<0,170>[`;\lambda']
\morphism(2500,550)|l|/=>/<0,200>[`;1_\id]

\place(-700,650)[\ \ ]
\place(3700,650)[(22)]
\efig
$$

$$
\bfig\scalefactor{.8}
\Vtriangle/`>`<-/<500,300>[{\bf B}_1 `{\bf B}_1`{\bf B}_2;`\Id_2`M]
\Vtriangle(0,700)/=`>`<-/<500,300>[{\bf A}_1`{\bf A}_1`{\bf A}_2;`\Id_2`M]

\morphism(0,1000)|l|<0,-700>[{\bf A}_1`{\bf B}_1;\id]
\morphism(500,700)|l|<0,-700>[{\bf A}_2`{\bf B}_2;\id]
\morphism(1000,1000)|r|<0,-700>[{\bf A}_1`{\bf B}_1;\id]

\morphism(250,400)|l|/=>/<0,200>[`;\tau_2]
\morphism(750,400)|l|/=>/<0,200>[`;\mu]
\morphism(500,800)|l|/=>/<0,170>[`;\rho']

\place(1500,650)[=]

\Vtriangle(2000,0)/=`>`<-/<500,300>[{\bf B}_1 `{\bf B}_1`{\bf B}_2;`\Id_2`M]
\square(2000,300)/=`>`>`=/<1000,700>[{\bf A}_1`{\bf A}_1`{\bf B}_1`{\bf B}_1;`\id`\id`]

\morphism(2500,100)|l|/=>/<0,170>[`;\rho']
\morphism(2500,550)|l|/=>/<0,200>[`;1_\id]

\place(-700,650)[\ \ ]
\place(3700,650)[(23)]
\efig
$$

\noindent Those involving $ m $
$$
\bfig\scalefactor{.8}

\Vtriangle/`>`<-/<500,300>[{\bf B}_3 `{\bf B}_1`{\bf B}_2;`M_{23}`M]

\Vtriangle(0,700)/`>`<-/<500,300>[{\bf C}_3`{\bf C}_1`{\bf C}_2;`M_{23}`M]

\Atriangle(0,1000)/<-`>`/<500,300>[{\bf C}_2`{\bf C}_3`{\bf C}_1;M_{12}`M`]

\morphism(0,1000)|l|<0,-700>[{\bf C}_3`{\bf B}_3;m]

\morphism(500,700)|l|<0,-700>[{\bf C}_2`{\bf B}_2;m]

\morphism(1000,1000)|r|<0,-700>[{\bf C}_1`{\bf B}_1;m]

\place(1500,650)[=]

\Vtriangle(2000,0)/`>`<-/<500,300>[{\bf B}_3 `{\bf B}_1`{\bf B}_2;`M_{23}`M]

\Atriangle(2000,300)/<-`>`/<500,300>[{\bf B}_2`{\bf B}_3`{\bf B}_1;M_{12}`M`]

\Atriangle(2000,1000)/<-`>`/<500,300>[{\bf C}_2`{\bf C}_3`{\bf C}_1;M_{12}`M`]

\morphism(2000,1000)|l|<0,-700>[{\bf C}_3`{\bf B}_3;m]

\morphism(2500,1300)|l|<0,-700>[{\bf C}_2`{\bf B}_2;m]

\morphism(3000,1000)|r|<0,-700>[{\bf C}_1`{\bf B}_1;m]

\morphism(250,400)|l|/=>/<0,200>[`;\chi_{23}]

\morphism(750,400)|l|/=>/<0,200>[`;\chi]

\morphism(500,950)|l|/=>/<0,200>[`;\kappa']

\morphism(2250,700)|l|/=>/<0,200>[`;\chi_{12}]

\morphism(2750,700)|l|/=>/<0,200>[`;\chi]

\morphism(2500,200)|l|/=>/<0,200>[`;\kappa']

\place(-700,650)[\ \ ]
\place(3700,650)[(24)]

\efig
$$

$$
\bfig\scalefactor{.8}

\Vtriangle/`>`<-/<500,300>[{\bf B}_1 `{\bf B}_1`{\bf B}_2;`\Id_1`M]

\Vtriangle(0,700)/=`>`<-/<500,300>[{\bf C}_1`{\bf C}_1`{\bf C}_2;`\Id_1`M]

\morphism(0,1000)|l|<0,-700>[{\bf C}_1`{\bf B}_1;m]

\morphism(500,700)|l|<0,-700>[{\bf C}_2`{\bf B}_2;m]

\morphism(1000,1000)|r|<0,-700>[{\bf C}_1`{\bf B}_1;m]

\morphism(250,400)|l|/=>/<0,200>[`;\delta_1]

\morphism(750,400)|l|/=>/<0,200>[`;\chi]

\morphism(500,800)|l|/=>/<0,170>[`;\lambda']

\place(1500,650)[=]

\Vtriangle(2000,0)/=`>`<-/<500,300>[{\bf B}_1 `{\bf B}_1`{\bf B}_2;``]

\square(2000,300)/=`>`>`=/<1000,700>[{\bf C}_1`{\bf C}_1`{\bf B}_1`{\bf B}_1;`m`m`]

\morphism(2500,100)|l|/=>/<0,170>[`;\lambda']

\morphism(2500,550)|l|/=>/<0,200>[`;1_m]

\place(-700,650)[\ \ ]
\place(3700,650)[(25)]
\efig
$$

$$
\bfig\scalefactor{.8}

\Vtriangle/`>`<-/<500,300>[{\bf B}_1 `{\bf B}_1`{\bf B}_2;`\Id_2`M]

\Vtriangle(0,700)/=`>`<-/<500,300>[{\bf C}_1`{\bf C}_1`{\bf C}_2;`\Id_2`M]

\morphism(0,1000)|l|<0,-700>[{\bf C}_1`{\bf B}_1;m]

\morphism(500,700)|l|<0,-700>[{\bf C}_2`{\bf B}_2;m]

\morphism(1000,1000)|r|<0,-700>[{\bf C}_1`{\bf B}_1;m]

\morphism(250,400)|l|/=>/<0,200>[`;\delta_2]

\morphism(750,400)|l|/=>/<0,200>[`;\chi]

\morphism(500,800)|l|/=>/<0,170>[`;\rho']

\place(1500,650)[=]

\Vtriangle(2000,0)/=`>`<-/<500,300>[{\bf B}_1 `{\bf B}_1`{\bf B}_2;``]

\square(2000,300)/=`>`>`=/<1000,700>[{\bf C}_1`{\bf C}_1`{\bf B}_1`{\bf B}_1;`m`m`]

\morphism(2500,100)|l|/=>/<0,170>[`;\rho']

\morphism(2500,550)|l|/=>/<0,200>[`;1_m]

\place(-700,650)[\ \ ]
\place(3700,650)[(26)]
\efig
$$

\noindent Those specific to $ \kappa $
$$
\bfig\scalefactor{.8}
\Vtriangle/`>`<-/<500,300>[{\bf D}_2 `{\bf B}_2`{\bf C}_2;`m_{23}`m]

\Vtriangle(0,700)/`>`<-/<500,300>[{\bf D}_1`{\bf B}_1`{\bf C}_1;`m_{23}`m]

\Atriangle(0,1000)/<-`>`/<500,300>[{\bf C}_1`{\bf D}_1`{\bf B}_1;m_{12}`m`]

\morphism(0,1000)|l|/<-/<0,-700>[{\bf D}_1`{\bf D}_2;M]

\morphism(500,700)|l|/<-/<0,-700>[{\bf C}_1`{\bf C}_2;M]

\morphism(1000,1000)|r|/<-/<0,-700>[{\bf B}_1`{\bf B}_2;M]

\place(1500,650)[=]

\Vtriangle(2000,0)/`>`<-/<500,300>[{\bf D}_2 `{\bf B}_2`{\bf C}_2;`m_{23}`m]

\Atriangle(2000,300)/<-`>`/<500,300>[{\bf C}_2`{\bf D}_2`{\bf B}_2;m_{12}`m`]

\Atriangle(2000,1000)/<-`>`/<500,300>[{\bf C}_1`{\bf D}_1`{\bf B}_1;m_{12}`m`]

\morphism(2000,1000)|l|/<-/<0,-700>[{\bf D}_1`{\bf D}_2;M]

\morphism(2500,1300)|l|/<-/<0,-700>[{\bf C}_1`{\bf C}_2;M]

\morphism(3000,1000)|r|/<-/<0,-700>[{\bf B}_1`{\bf B}_2;M]

\morphism(250,400)|l|/=>/<0,200>[`;\chi_{23}]
\morphism(750,400)|l|/=>/<0,200>[`;\chi]
\morphism(500,950)|l|/=>/<0,200>[`;\kappa]

\morphism(2250,700)|l|/=>/<0,200>[`;\chi_{12}]
\morphism(2750,700)|l|/=>/<0,200>[`;\chi]
\morphism(2500,200)|l|/=>/<0,200>[`;\kappa]

\place(-700,650)[\ \ ]
\place(3700,650)[(27)]

\efig
$$

$$
\bfig\scalefactor{.8}
\Vtriangle/`>`<-/<500,300>[{\bf D}_0 `{\bf B}_0`{\bf C}_0;`m_{23}`m]

\Vtriangle(0,700)/`>`<-/<500,300>[{\bf D}_1`{\bf B}_1`{\bf C}_1;`m_{23}`m]

\Atriangle(0,1000)/<-`>`/<500,300>[{\bf C}_1`{\bf D}_1`{\bf B}_1;m_{12}`m`]

\morphism(0,1000)|l|/<-/<0,-700>[{\bf D}_1`{\bf D}_0;\Id]

\morphism(500,700)|l|/<-/<0,-700>[{\bf C}_1`{\bf C}_0;\Id]

\morphism(1000,1000)|r|/<-/<0,-700>[{\bf B}_1`{\bf B}_0;\Id]

\place(1500,650)[=]

\Vtriangle(2000,0)/`>`<-/<500,300>[{\bf D}_0 `{\bf B}_0`{\bf C}_0;`m_{23}`m]

\Atriangle(2000,300)/<-`>`/<500,300>[{\bf C}_0`{\bf D}_0`{\bf B}_0;m_{12}`m`]

\Atriangle(2000,1000)/<-`>`/<500,300>[{\bf C}_1`{\bf D}_1`{\bf B}_1;m_{12}`m`]

\morphism(2000,1000)|l|/<-/<0,-700>[{\bf D}_1`{\bf D}_0;\Id]

\morphism(2500,1300)|l|/<-/<0,-700>[{\bf C}_1`{\bf C}_0;\Id]

\morphism(3000,1000)|r|/<-/<0,-700>[{\bf B}_1`{\bf B}_0;\Id]

\morphism(250,400)|l|/=>/<0,200>[`;\delta_{23}]
\morphism(750,400)|l|/=>/<0,200>[`;\delta]
\morphism(500,950)|l|/=>/<0,200>[`;\kappa]

\morphism(2250,700)|l|/=>/<0,200>[`;\delta]
\morphism(2750,700)|l|/=>/<0,200>[`;\delta]
\morphism(2500,200)|l|/=>/<0,200>[`;\kappa]

\place(-700,650)[\ \ ]
\place(3700,650)[(28)]

\efig
$$

\noindent Those specific to $ \lambda $
$$
\bfig\scalefactor{.8}

\Vtriangle/`>`<-/<500,300>[{\bf B}_2 `{\bf B}_2`{\bf C}_2;`\id_1`m]

\Vtriangle(0,700)/=`>`<-/<500,300>[{\bf B}_1`{\bf B}_1`{\bf C}_1;`\id_1`m]

\morphism(0,1000)|l|/<-/<0,-700>[{\bf B}_1`{\bf B}_2;M]

\morphism(500,700)|l|/<-/<0,-700>[{\bf C}_1`{\bf C}_2;M]

\morphism(1000,1000)|r|/<-/<0,-700>[{\bf B}_1`{\bf B}_2;M]

\morphism(250,400)|l|/=>/<0,200>[`;\mu_1]

\morphism(750,400)|l|/=>/<0,200>[`;\chi]

\morphism(500,800)|l|/=>/<0,170>[`;\lambda]

\place(1500,650)[=]

\Vtriangle(2000,0)/=`>`<-/<500,300>[{\bf B}_2 `{\bf B}_2`{\bf C}_2;`\id_1`m]

\square(2000,300)/=`<-`<-`=/<1000,700>[{\bf B}_1`{\bf B}_1`{\bf B}_2`{\bf B}_2;`M`M`]

\morphism(2500,100)|l|/=>/<0,170>[`;\lambda]

\morphism(2500,550)|l|/=>/<0,200>[`;1_\mu]

\place(-700,650)[\ \ ]
\place(3700,650)[(29)]
\efig
$$

$$
\bfig\scalefactor{.8}

\Vtriangle/`>`<-/<500,300>[{\bf B}_0 `{\bf B}_0`{\bf C}_0;`\id_1`m]

\Vtriangle(0,700)/=`>`<-/<500,300>[{\bf B}_1`{\bf B}_1`{\bf C}_1;`\id_1`m]

\morphism(0,1000)|l|/<-/<0,-700>[{\bf B}_1`{\bf B}_0;\Id]

\morphism(500,700)|l|/<-/<0,-700>[{\bf C}_1`{\bf C}_0;\Id]

\morphism(1000,1000)|r|/<-/<0,-700>[{\bf B}_1`{\bf B}_0;\Id]

\morphism(250,400)|l|/=>/<0,200>[`;\tau_1]

\morphism(750,400)|l|/=>/<0,200>[`;\delta]

\morphism(500,800)|l|/=>/<0,170>[`;\lambda]

\place(1500,650)[=]

\Vtriangle(2000,0)/=`>`<-/<500,300>[{\bf B}_0 `{\bf B}_0`{\bf C}_0;`\id_1`m]

\square(2000,300)/=`<-`<-`=/<1000,700>[{\bf B}_1`{\bf B}_1`{\bf B}_0`{\bf B}_0;`\Id`\Id`]

\morphism(2500,100)|l|/=>/<0,170>[`;\lambda]

\morphism(2500,550)|l|/=>/<0,200>[`;1_\Id]

\place(-700,650)[\ \ ]
\place(3700,650)[(30)]
\efig
$$

\noindent Those specific to $ \rho $
$$
\bfig\scalefactor{.8}

\Vtriangle/`>`<-/<500,300>[{\bf B}_2 `{\bf B}_2`{\bf C}_2;`\id_2`m]

\Vtriangle(0,700)/=`>`<-/<500,300>[{\bf B}_1`{\bf B}_1`{\bf C}_1;`\id_2`m]

\morphism(0,1000)|l|/<-/<0,-700>[{\bf B}_1`{\bf B}_2;M]

\morphism(500,700)|l|/<-/<0,-700>[{\bf C}_1`{\bf C}_2;M]

\morphism(1000,1000)|r|/<-/<0,-700>[{\bf B}_1`{\bf B}_2;M]

\morphism(250,400)|l|/=>/<0,200>[`;\mu_2]

\morphism(750,400)|l|/=>/<0,200>[`;\chi]

\morphism(500,800)|l|/=>/<0,170>[`;\rho]

\place(1500,650)[=]

\Vtriangle(2000,0)/=`>`<-/<500,300>[{\bf B}_2 `{\bf B}_2`{\bf C}_2;`\id_2`m]

\square(2000,300)/=`<-`<-`=/<1000,700>[{\bf B}_1`{\bf B}_1`{\bf B}_2`{\bf B}_2;`M`M`]

\morphism(2500,100)|l|/=>/<0,170>[`;\rho]

\morphism(2500,550)|l|/=>/<0,200>[`;1_\mu]

\place(-700,650)[\ \ ]
\place(3700,650)[(31)]
\efig
$$

$$
\bfig\scalefactor{.8}

\Vtriangle/`>`<-/<500,300>[{\bf B}_0 `{\bf B}_0`{\bf C}_0;`\id_2`m]

\Vtriangle(0,700)/=`>`<-/<500,300>[{\bf B}_1`{\bf B}_1`{\bf C}_1;`\id_2`m]

\morphism(0,1000)|l|/<-/<0,-700>[{\bf B}_1`{\bf B}_0;\Id]

\morphism(500,700)|l|/<-/<0,-700>[{\bf C}_1`{\bf C}_0;\Id]

\morphism(1000,1000)|r|/<-/<0,-700>[{\bf B}_1`{\bf B}_0;\Id]

\morphism(250,400)|l|/=>/<0,200>[`;\tau_2]

\morphism(750,400)|l|/=>/<0,200>[`;\delta]

\morphism(500,800)|l|/=>/<0,170>[`;\rho]

\place(1500,650)[=]

\Vtriangle(2000,0)/=`>`<-/<500,300>[{\bf B}_0 `{\bf B}_0`{\bf C}_0;`\id_2`m]

\square(2000,300)/=`<-`<-`=/<1000,700>[{\bf B}_1`{\bf B}_1`{\bf B}_0`{\bf B}_0;`\Id`\Id`]

\morphism(2500,100)|l|/=>/<0,170>[`;\rho]

\morphism(2500,550)|l|/=>/<0,200>[`;1_\Id]

\place(-700,650)[\ \ ]
\place(3700,650)[(32)]
\efig
$$
  
Let us call the above structure {\em double pseudocategory object}.

\begin{theorem}\label{equivalence} The following three structures are the same, i.e.~they are specified by the same data and satisfy the same conditions.

\noindent (a) Double pseudocategories in $\Cat $.

\noindent (b) Horizontal intercategories, i.e.~pseudocategories
$$
 {\mathbb C}\  \threepppp/>`>`>/<400>^{p_1} |{m} _{p_2} \ {\mathbb B} \ \three/>`<-`>/^{\partial_0} | {\id}_{\partial_1} \ {\mathbb A} \mbox{\quad in\ } \Dlax
$$

\noindent (c) Vertical intercategories, i.e.~pseudocategories
$$
 {\mathbb X}_2\  \threepppp/>`>`>/<400>^{P_1} |{M} _{P_2} \ {\mathbb X}_1 \ \three/>`<-`>/^{D_0} | {\Id}_{D_1} \ {\mathbb X}_0  \mbox{\quad in \ } \Dcolax
$$

\end{theorem}

\begin{proof} (a) $\Leftrightarrow $ (b) The double categories $ {\mathbb A}, {\mathbb B}, {\mathbb C} $, when expressed as truncated simplicial categories, give the columns of $ (*) $, the structural isomorphisms $ \kappa', \lambda', \rho' $ and the pullbacks (2), (3), (4).

That the $ \partial_i $ are strict double functors is equivalent to conditions (5), (6), (7) and (8). Condition (6) is to ensure that $ \partial^2_i $ is the right morphism.

That
$$
\bfig

\square[{\mathbb C} `{\mathbb B} `{\mathbb B} `{\mathbb A};p_2 `p_1 `\partial_0 `\partial_1]

\efig
$$
is a pullback is equivalent to having the pullbacks (1) together with (9), (10), (11) and (12).

For $ \id : {\mathbb A} \to {\mathbb B} $ to be a lax functor it must preserve domain and codomain (13) and satisfy (14) to ensure that $ \id^2 $ is the right morphism. The identity comparison is given by $ \tau $ of (17) and the composition comparison by $ \mu $ of (18).

Similarly $ m : {\mathbb C} \to {\mathbb B} $ is a lax functor if and only if it satisfies (15) and (16) with the laxity morphisms given by $ \delta $ (19) and $ \chi $ (20).

With the structural isomorphisms $ \kappa, \lambda, \rho $ of $ {\mathbb C}\  \threepppp/>`>`>/<400>^{} |{} _{} \ {\mathbb B} \ \three/>`<-`>/^{} | {}_{} \ {\mathbb A} $ we see that the data and functor commutativities for a vertical intercategory and a double pseudocategory in $\Cat $ are the same ((1)-(20)). The equations on the natural transformations also correspond.

The coherence conditions for $ \id $ are: associativity (21), left unit law (22) and right unit law (23).

The coherence conditions for $ m $ are: associativity (24), left unit law (25) and right unit law (26).

For $ \kappa $ to be a transformation we need (27), (28); for $ \lambda $, we need (29), (30) and for $ \rho $, (31), (32).

This shows that a horizontal intercategory is the same as a double pseudocategory in $ \Cat $. To show that a vertical intercategory is also the same is similar although the conditions correspond in a different order. We give them here for completeness.

That $ D_i $ are strict functors correspond to (5), (9), (13), (15).

That
$$
\bfig

\square[{\mathbb X}_2 ` {\mathbb X}_1 ` {\mathbb X}_1 `{\mathbb X}_0;P_2 `P_1 `D_0 `D_1]

\efig
$$
is a pullback corresponds to the pullbacks (2), (3), (4) and commutativities (6), (10), (14), (16).

$ \Id $ is colax corresponds to (7), (11) with the structure maps $ \tau $ (17) and $\delta $ (19).

$ M $ is colax corresponds to (8), (12) and the structure maps $ \mu $ (18) and $ \chi $ (20).

The coherence conditions for $ \Id $ are (28), (30), (32) and those for $ M $ are (27), (29), (31).

For $ \kappa' $ to be a transformation we need (21), (24); for $ \lambda' $ we need (22), (25) and for $ \rho' $ we need (23), (26).

\end{proof}

The following table summarizes how the four different presentations of intercategory are related.

\ 

\noindent \begin{tabular}{|l|l|l|l|} \hline 
Object of $\ic A$  \rule{0cm}{15pt}    &  Object of $ {\bf A}_0 $       & Object of $ {\mathbb A} $       & Object of $ {\mathbb X}_0 $ \\[2pt]
Transversal arrow      & Morphism of $ {\bf A}_0 $   & Horizontal arrow of $ {\mathbb A} $  & Horizontal arrow of $ {\mathbb X}_0 $\\[2pt] 
Horizontal arrow   &  Object of $ {\bf B}_0 $  & Object of $ {\mathbb B} $ &  Vertical arrow of ${\mathbb X}_0 $ \\[2pt]
Vertical arrow  & Object of $ {\bf A}_1 $   &  Vertical arrow of $ {\mathbb A} $    & Object of $ {\mathbb X}_1 $ \\[2pt]
Horizontal cell   &   Morphism of $ {\bf B}_0 $    &   Horizontal arrow of $ {\mathbb B} $    &  Double cell of $ {\mathbb X}_0 $\\[2pt]
Vertical cell   &  Morphism of $ {\bf A}_1 $     &   Double cell of $ {\mathbb A} $   &   Horizontal arrow of $ {\mathbb X}_1 $\\[2pt]
Basic cell    &   Object of $ {\bf B}_1 $    &  Vertical arrow of $ {\mathbb B} $    &  Vertical arrow of $ {\mathbb X}_1 $\\[2pt]
Cube      & Morphism of $ {\bf B}_1 $  & Double cell of $ {\mathbb B} $   & Double cell of $ {\mathbb X}_1 $\\ [2pt]\hline
\end{tabular}

\ 

\ 

There are eight symmetries for intercategories, generated by three dualities, h, v, tr, just like for double categories. As the basic unit of structure in an intercategory is the cube, one might expect 48, as we would have for triple categories. But the transversal direction is special and must be invariant, which cuts the possibilities down to 16. However, since the interchanger $\chi $ is directed, only half of these are valid symmetries. In fact there are two dual notions of intercategory: the one we are using here in which
$$
\chi : \dfrac{\ \ | \ \ }{\ \ | \ \ }\  \to \ \dfrac{\ \ }{\ \ }\left| \dfrac{\ \ }{\ \ }\right.
$$
that we will call {\em right intercategory}, and the one in which $ \chi $ goes in the opposite direction called {\em left intercategory}.

The dualities are determined by what they do to the horizontal and vertical direction:

\noindent -- h reverses the horizontal direction while maintaining the vertical and transversal,

\noindent -- v reverses the vertical direction while maintaining the horizontal and transversal,

\noindent -- tr switches the horizontal and vertical directions and reverses the transversal. Because the interchangers are directed from a vertical composite of horizontal composites to the other way around, if we switch horizontal and vertical we must also reverse the direction of transversal arrows.

\noindent The other symmetries are combinations of these.


\section{Coherence in the $2$-dimensional notation} \label{depiction}

Apart from the interchangers $ \chi$, $ \delta $, $\mu $, $\tau$ there are the associativity and unit isomorphisms for $ {\ic A} $.
$$
\kappa  : \alpha | (\beta | \gamma) \to (\alpha | \beta) | \gamma,\quad   \lambda : \id_v | \alpha \to \alpha,\quad \rho : \alpha | \id_w \to \alpha
$$
and
$$
\kappa' : \dfrac{\alpha}{\frac{\beta}{\gamma}} \to \dfrac{\frac{\alpha}{\beta}}{\gamma},\quad \lambda' : \frac{\Id_h}{\alpha} \to \alpha ,\quad  \rho' : \frac{\alpha}{\Id_{\ov{h}}} \to \alpha .
$$
We usually omit the $ \kappa $ and $ \kappa' $ (leaving them understood), but keep the left and right unit isomorphisms.

The two-dimensional notation makes the overall structure easier to understand. Conditions (1)-(16) in the definition of double pseudocategory object are all bookkeeping for domains, codomains and composable pairs, and are implicit in the geometric representation. (17)-(20) give the interchangers, whose naturality is also implicit.

Conditions (21)-(32) are as follows.
$$
\bfig

\square<500,800>[\dfrac{\id_v}{\dfrac{\id_{\ov{v}}}{\id_{\ovv{v}}}}
`\dfrac{\id_{\frac{v}{\ov v}}}{\id_\ovv{v}}
`\dfrac{\id_v}{\id_{\frac{\ov{v}}{\ovv{v}}}}
`\id_{\frac{v}{\frac{\ov{v}}{\ovv{v}}}};
\frac{\mu}{\id_{\ovv{v}}}
`\frac{\id_v}{\mu}
`\mu`\mu]

\place(250,400)[\scriptstyle (21)]

\square(1000,150)/>`>`>`<-/[\dfrac{\Id_{\id_A}}{\id_v}
    `\dfrac{\id_{\Id_A}}{\id_v}
   `\id_v
    `\id_{\frac{\Id_A}{v}};
\frac{\tau}{\id_v}
`\lambda'
`\mu
`\id_{\lambda'}]

\place(1250,400)[\scriptstyle (22)]

\square(2000,150)/>`>`>`<-/[\dfrac{\id_v }{\Id_{\id_{\ov{A}}}}
      `\dfrac{\id_v }{\id_{\Id_{\ov{A}}}}
   `\id_v
   `\id_{\frac{v}{\Id_{\ov{A}}}};
\frac{\id_v}{\tau}
`\rho'
`\mu
`\id_{\rho'}]

\place(2250,400)[\scriptstyle (23)]

\efig
$$

$$
\bfig

\square<500,1000>[\dfrac{\alpha | \beta}{\dfrac{\gamma |\delta}{\epsilon|\phi}}
`\dfrac{\alpha | \beta}{\left.\dfrac{\gamma}{\epsilon}\right|\dfrac{\delta}{\phi}}
       `\dfrac{\left.\dfrac{\alpha}{\gamma}\right|\dfrac{\beta}{\delta}}{\epsilon |\phi}
        `\left.\dfrac{\alpha}{\dfrac{\gamma}{\epsilon}}\right|\dfrac{\beta}{\dfrac{\delta}{\phi}};
\dfrac{\alpha | \beta}{\chi}   `\dfrac{\chi}{\epsilon | \phi} `\chi `\chi]

\place(250,500)[\scriptstyle (24)]

\efig
$$

$$
\bfig

\square/>`>`<-`>/<500,800>[\dfrac{\Id_{h|k}}{\alpha | \beta}
   `\alpha | \beta
      `\dfrac{\Id_h | \Id_k}{\alpha | \beta}
   `\left.\dfrac{\Id_h}{\alpha}\right|\dfrac{\Id_k}{\beta};
\lambda' ` \dfrac{\delta}{\alpha | \beta} `\lambda' | \lambda'`\chi]

\place(250,400)[\scriptstyle (25)]

\square(1500,0)/>`>`<-`>/<500,800>[\dfrac{\alpha | \beta}{\Id_{\ov{h}|\ov{k}}} 
   `\alpha | \beta
   `\dfrac{\alpha | \beta}{\Id_{\ov{h}} | \Id_{\ov{k}}}
   `\left.\dfrac{\alpha}{\Id_{\ov{h}}}\right|\dfrac{\beta}{\Id_{\ov{k}}};
\rho' `\dfrac{\alpha | \beta}{\delta} `\rho' | \rho' `\chi]

\place(1750,400)[\scriptstyle (26)]

\efig
$$

$$
\bfig

\square<800,800>[\dfrac{\alpha | \beta | \gamma}{\delta | \epsilon | \phi}
       `\left.\dfrac{\alpha |\beta}{\delta | \epsilon}\right|\dfrac{\gamma}{\phi}
       `\left.\dfrac{\alpha}{\delta}\right|\dfrac{\beta | \gamma}{\epsilon | \phi}
      `\left.\dfrac{\alpha}{\delta}\right|\left.\dfrac{\beta}{\epsilon}\right|\dfrac{\gamma}{\phi};
 \chi `\chi `\chi\left|\dfrac{\gamma}{\phi}\right.
`\left.\dfrac{\alpha}{\delta}\right| \chi]

\place(400,400)[\scriptstyle (27)]

\efig
$$

$$
\bfig

\square<800,500>[\Id_{h|k|l}
   ` \Id_h | \Id_{k|l} 
   `\Id_{h|k} | \Id_l 
    `\Id_h | \Id_k | \Id_l; 
\delta 
` \delta 
`\Id_h | \delta 
` \delta | \Id_l]

\place(400,250)[\scriptstyle (28)]

\efig
$$

$$
\bfig

\square/>`>`>`<-/<500,650>[\dfrac{\id_v | \alpha}{\id_{\ov{v}} | \beta}
  `\left.\dfrac{\id_v}{\id_{\ov{v}}}\right| \dfrac{\alpha}{\beta}
 `\dfrac{\alpha}{\beta}
`\id_{\frac{v}{\ov{v}}}\left|\dfrac{\alpha}{\beta}\right.;
\chi ` \frac\lambda\lambda `\mu \left|\frac{\alpha}{\beta}\right.`\lambda]

\place(250,325)[\scriptstyle (29)]

\square(1500,0)/>`>`>`<-/<500,650>[\Id_{\id_A | h}
` \Id_{\id_A} | \Id_h 
`\Id_h 
`\id_{\Id_A} | \Id_h; 
\delta 
` \Id_\lambda 
` \tau | \Id 
` \lambda]

\place(1750,325)[\scriptstyle (30)]

\efig
$$

$$
\bfig

\square/>`>`>`<-/<500,650>[\dfrac{\alpha | \id_w}{\beta | \id_{\ov{w}}} 
   `\left.\dfrac{\alpha}{\beta}\right|\dfrac{\id_w}{\id_{\ov{w}}}
  `\dfrac{\alpha}{\beta}
`\left.\dfrac{\alpha}{\beta}\right| \id_{\frac{w}{\ov{w}}};
\chi 
` \frac\rho\rho 
` \left.\frac{\alpha}{\beta}\right| \mu
`\rho]

\place(250,325)[\scriptstyle (31)]

\square(1500,0)/>`>`>`<-/<500,650>[\Id_{h | \id_B} 
` \Id_h | \Id_{\id_B} 
`\Id_h 
`\Id_h | \id_{\Id_B};
\delta 
` \Id_\rho 
` \Id | \tau 
` \rho]

\place(1750,325)[\scriptstyle (32)]

\efig
$$


\section{Morphisms}\label{mor}

As mentioned in Section 3, by considering intercategories as pseudo-category objects in $ \Dlax $ we get two notions of morphism between them corresponding to internal lax and colax functors. As we see below, we also get cells relating them. We get a similar situation looking at pseudocategories in $\Dcolax $. We are now in a position to analyze this further.

First, we generalize our double category $ {\mathbb D}{\rm bl} $ of \cite{GP}.

\begin{theorem} \label{PC} For any $2$-category $ {\cal A} $ we get a strict double category $ {\mathbb P}s{\mathbb C}{\rm at} ({\cal A}) $, whose objects are pseudo-category objects in ${\cal A} $, with lax functors as horizontal arrows and colax functors as vertical ones and suitable double cells.

\end{theorem}

\begin{proof} It is clear that we have a strictly associative and unitary composition of lax functors, and colax ones too.

We must define double cells, which is just a question of expressing the definition of \cite{GP} diagramatically. A double cell $ \pi $ has a boundary
$$
\bfig

\square/>`@{>}|{\bb}`@{>}|{\bb}`>/[{\bf A} `{\bf B} `{\bf C} `{\bf D};F `U`V `G]

\place(250,250)[\scriptstyle \pi]

\efig
$$
with ${\bf A}, {\bf B}, {\bf C}, {\bf D} $ pseudocategories in $ {\cal A} $, $ F, G $ lax functors and $ U, V $ colax functors. $ {\bf A} $ consists, in part, of objects and arrows
$$
A_2 \threepppp/>`>`>/<400>^{p_1} |{m} _{p_2} A_1 \three/>`<-`>/^{\partial_0} | {\id}_{\partial_1} A_0
$$
and similarly for $ {\bf B}, {\bf C}, {\bf D} $. Similarly $ F $ consists of arrows
$$
\bfig

\square/>`>`>`>/[A_2 `A_1`B_2`B_1;`F_2`F_1`]

\square|allb|/@{>}@<-3pt>`` `@{>}@<-3pt>/[A_2 `A_1`B_2`B_1;```]

\square|allb|/@{>}@<3pt>```@{>}@<3pt>/[A_2 `A_1`B_2`B_1;```]

\square(500,0)/<-`>`>`<-/[A_1 `A_0`B_1`B_0;``F_0`]

\square(500,0)|allb|/@{>}@<-3pt>```@{>}@<-3pt>/[A_1 `A_0`B_1`B_0;```]

\square(500,0)|allb|/@{>}@<3pt>```@{>}@<3pt>/[A_1 `A_0`B_1`B_0;```]

\efig
$$
commuting with $ \partial_0, \partial_1$ (so $ F_2 $ is determined by $ F_0 $ and $ F_1 $). As in \cite{GP} we denote the laxity cells by $ \underline{F} $, or more explicitly
$$
\bfig

\square[A_0`A_1`B_0`B_1;\id `F_0 `F_1 `\id]

\morphism(230,200)/=>/<100,100>[`;\underline{F}_0]

\place(850,250)[\mbox{and}]

\square(1200,0)[A_2 `A_1`B_2`B_1;m `F_2`F_1`m]

\morphism(1430,200)/=>/<100,100>[`;\underline{F}_2]

\efig
$$
Similar notation applies to $ G, U, V $. The cell $ \pi $ then consists of three cells, $ \pi_0, \pi_1, \pi_2 $
$$
\bfig

\square[A_i`B_i`C_i`D_i;F_i`U_i`V_i`G_i]

\morphism(320,290)/=>/<-100,-100>[`;\pi_i]

\efig
$$
which preserve domains and codomains (so $ \pi_2 $ is determined by $ \pi_0 $ and $\pi_1 $) and satisfy
$$
\bfig

\square(0,300)[A_i ` B_i `C_i `D_i; F_i` U_i`V_i `G_i]

\square(300,0)/``>`>/[` B_1 `C_1 `D_1;` `V_1 `G_1]

\morphism(500,800)<300,-300>[B_i ` B_1;]

\morphism(0,300)|b|<300,-300>[C_i`C_1;]

\morphism(500,300)<300,-300>[D_i `D_1;]

\morphism(300,600)/=>/<-100,-100>[`;\pi_i]

\morphism(450,170)/=>/<-100,-100>[`;\underline{G}_i]

\morphism(750,400)/=>/<-100,-100>[`;\underline{V}_i]

\place(1100,400)[=]

\square(1400,300)/>`>``/[A_i ` B_i `C_i `;F_i `U_i ` `]

\square(1700,0)[A_1` B_1 `C_1 `D_1; G_1 ` U_1 `V_1 `G_1]

\morphism(1400,800)|a|<300,-300>[A_i` A_1;]

\morphism(1900,800)<300,-300>[B_i ` B_1;]

\morphism(1400,300)|b|<300,-300>[C_i`C_1;]

\morphism(2050,250)/=>/<-100,-100>[`;\pi_1]

\morphism(1650,400)/=>/<-100,-100>[`;\underline{U}_i]

\morphism(1900,700)/=>/<-100,-100>[`;\underline{F}_i]

\efig
$$
where the transversal arrows are $ \id$'s for $ i = 0 $ and $m$'s for $ i = 2 $. The $ \pi_i $ can be pasted horizontally and vertically and associativity, unit laws and interchange hold strictly. Note that the above conditions are commutative cubes of $2$-cells whose front is $ \pi_1 $ and back $ \pi_0 $ and $ \pi_2 $ respectively. So they may be pasted horizontally and vertically as well, and the pasted $\pi$'s still satisfy these conditions.

\end{proof}

What this gives us when applied to the $2$-category $ \Dlax$ are two kinds of morphism of intercategory and a corresponding notion of double cell. We can also apply it to the $2$-category $\Dcolax $ to get the same kind of thing. To see how they are related we must express them in more elementary terms.  First let's examine lax functors
$$
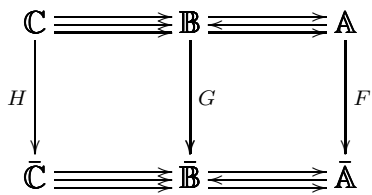
\bfig

\square/>`>`>`>/[{\mathbb C} ` {\mathbb B} `\ov{\mathbb C} `\ov{\mathbb B};`H`G`]

\square|allb|/@{>}@<-3pt>`` `@{>}@<-3pt>/[{\mathbb C} ` {\mathbb B} `\ov{\mathbb C} `\ov{\mathbb B};```]

\square|allb|/@{>}@<3pt>```@{>}@<3pt>/[{\mathbb C} ` {\mathbb B} `\ov{\mathbb C} `\ov{\mathbb B};```]

\square(500,0)/<-`>`>`<-/[{\mathbb B} `{\mathbb A} `\ov{\mathbb B} `\ov{\mathbb A};``F`]

\square(500,0)|allb|/@{>}@<-3pt>```@{>}@<-3pt>/[{\mathbb B} `{\mathbb A} `\ov{\mathbb B} `\ov{\mathbb A};```]

\square(500,0)|allb|/@{>}@<3pt>```@{>}@<3pt>/[{\mathbb B} `{\mathbb A} `\ov{\mathbb B} `\ov{\mathbb A};```]

\efig
$$
of pseudo-category objects in $ \Dlax $. $ F, G, H $ are lax functors as they come from $ \Dlax $ but we also have laxity transformations
$$
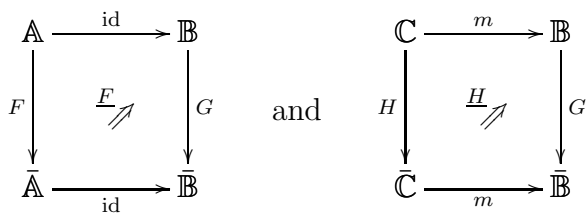
\bfig
\square[{\mathbb A} `{\mathbb B} `\ov{\mathbb A} `\ov{\mathbb B}; \id ` F ` G `\id]
\morphism(250,200)/=>/<100,100>[`;\underline{F}]

\place(850,250)[\mbox{and}]

\square(1200,0)[{\mathbb C} `{\mathbb B} `\ov{\mathbb C} `\ov{\mathbb B};m` H `G `m]

\morphism(1450,200)/=>/<100,100>[`;\underline{H}]

\efig
$$
$ F $ is given by functors $ F_i : {\bf A}_i \to \ov{{\bf A}_i} $, $ i = 0, 1, 2 $ and the same for $ G $ and $ H $, and all domains and codomains are preserved. Let's denote the corresponding morphism of intercategories by $ \Phi : {\ic A} \to \ov{\ic A} $.

Thus $ \Phi $ takes objects of $ {\ic A} $ to objects of $ \ov{\ic A} $ and transversal arrows to transversal arrows, given by $ F_0 $. It takes vertical arrows and vertical cells to similar ones in $ \ov{\ic A} $, given by $ F_1 $. Similarly it takes horizontal arrows and cells to the same kind in $ \ov{\ic A} $ (given by $ G_0 $) and basic cells and cubes to like ones in $ \ov{\ic A} $ (given by $ G_1 $). All domains and codomains are respected, as illustrated by
$$
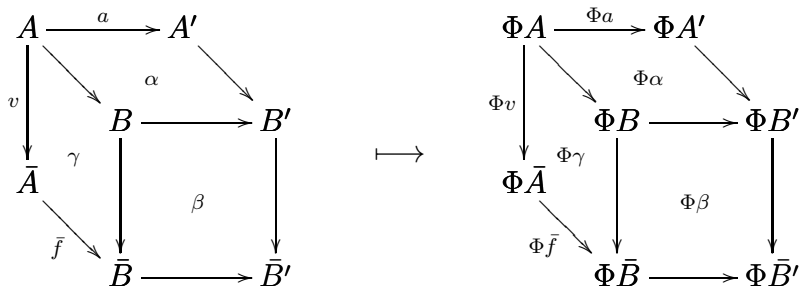
\bfig

\square(0,300)/>`>``/[A `A' `\ov{A}`;a ` v ` `]

\square(300,0)[B ` B' `\ov{B} `\ov{B'}; `  ``]

\morphism(0,800)|b|<300,-300>[A` B;]

\morphism(500,800)<300,-300>[A' ` B';]

\morphism(0,300)|b|<300,-300>[\ov{A}`\ov{B};\ov{f}]

\place(550,250)[\scriptstyle \beta]

\place(150,400)[\scriptstyle \gamma]

\place(400,650)[\scriptstyle \alpha]

\place(1200,400)[\longmapsto]

\square(1600,300)/>`>``/[\Phi A ` \Phi A' `\Phi\ov{A} `;\Phi a ` \Phi v ` `]

\square(1900,0)[\Phi B ` \Phi B' `\Phi\ov{B} `\Phi\ov{B'}; `  ` `]

\morphism(1600,800)|b|<300,-300>[\Phi A` \Phi B;]

\morphism(2100,800)<300,-300>[\Phi A' ` \Phi B';]

\morphism(1600,300)|b|<300,-300>[\Phi\ov{A}`\Phi\ov{B};\Phi\ov{f}]

\place(2150,250)[\scriptstyle \Phi \beta]

\place(1750,400)[\scriptstyle \Phi \gamma]

\place(2000,650)[\scriptstyle \Phi\alpha]

\efig
$$
Transversal composition (of arrows, horizontal and vertical cells) is preserved strictly. For vertical composition we are given comparison cells, coming from $ F $ and $ G $. For an object $ A $, a vertical cell
$$
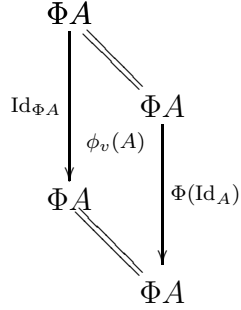
\bfig

\morphism(0,600)|l|<0,-600>[\Phi A`\Phi A;\Id_{\Phi A}]

\morphism(0,600)/=/<300,-300>[\Phi A`\Phi A;]

\morphism(0,0)/=/<300,-300>[`\Phi A;]

\morphism(300,300)|r|<0,-500>[`;\Phi(\Id_A)]

\place(150,200)[\scriptstyle \phi_v(A)]

\efig
$$
and for a horizontal arrow $ a : A \to A' $ a cube
$$
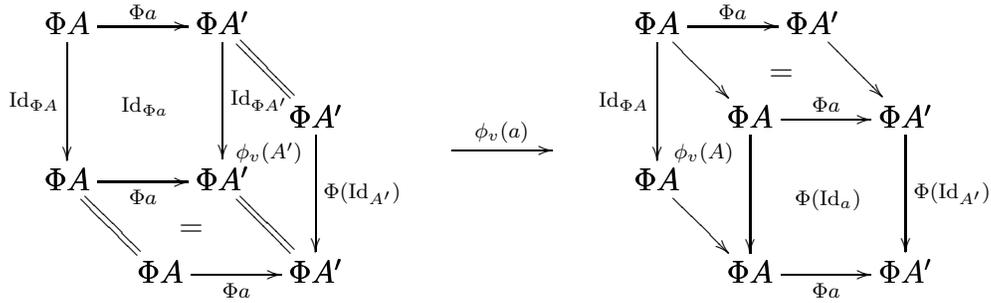
\bfig

\square(0,300)[\Phi A ` \Phi A' ` \Phi A `\Phi A';\Phi a `\Id_{\Phi A} `\Id_{\Phi A'} `\Phi a]

\square(300,0)/``>`>/[` \Phi A' `\Phi A `\Phi A';` `\Phi(\Id_{A'}) `\Phi a]

\morphism(500,800)/=/<300,-300>[\Phi A'` \Phi A';]

\morphism(0,300)|b|/=/<300,-300>[\Phi A`\Phi A;]

\morphism(500,300)/=/<300,-300>[\Phi A' `\Phi A';]

\place(250,550)[\scriptstyle \Id_{\Phi a}]

\place(400,150)[=]

\place(650,400)[\scriptstyle \phi_v (A')]

\place(1400,400)[\to^{\phi_v(a)}]

\square(1900,300)/>`>``/[\Phi A ` \Phi A' `\Phi A`;\Phi a ` \Id_{\Phi A} ` `]

\square(2200,0)[\Phi A ` \Phi A'`\Phi A `\Phi A'; \Phi a `  `\Phi(\Id_{A'}) `\Phi a]

\morphism(1900,800)|b|<300,-300>[\Phi A` \Phi A;]

\morphism(2400,800)<300,-300>[\Phi A' ` \Phi A';]

\morphism(1900,300)|b|<300,-300>[\Phi A`\Phi A;]

\place(2450,250)[\scriptstyle \Phi (\Id_a)]

\place(2050,400)[\scriptstyle \phi_v(A)]

\place(2300,650)[=]

\efig
$$
Following our simplified notation of Section 4, we write both of these as
$$
\bfig

\place(500,0)[\phi_v : \Id_\Phi \to \Phi (\Id).]

\place(2000,0)[(1)]

\place(-1000,0)[\ ]

\efig
$$
For vertical arrows $ A \todo{v} \ov{A} \todo{\ov{v}} \ovv{A} $ we are given vertical cells
$$
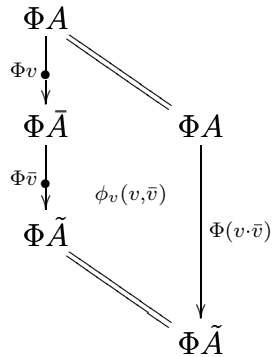
\bfig

\morphism(0,700)|l|/@{>}|{\bb}/<0,-350>[\Phi A `\Phi\ov{A};\Phi v]

\morphism(0,350)/@{>}|{\bb}/<0,-350>[\Phi \ov{A}`\Phi \ovv{A};\Phi\ov{v}]

\morphism(0,700)/=/<500,-350>[\Phi A`\Phi A;]

\morphism(0,0)/=/<500,-350>[\Phi\ovv{A}`\Phi\ovv{A};]

\morphism(500,350)|r|/>/<0,-700>[\Phi A`\Phi\ovv{A};\Phi(v\cdot \ov{v})]

\place(275,150)[\scriptstyle \phi_v (v,\ov{v})]

\efig
$$
and for basic cells
$$
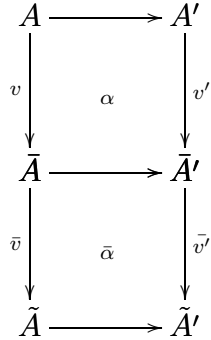
\bfig

\square[\ov{A} `\ov{A'} `\ovv{A} `\ovv{A'};`\ov{v}`\ov{v'}`]

\square(0,500)[A`A'`\ov{A}`\ov{A'};`v`v'`]

\place(250,250)[\scriptstyle \ov{\alpha}]

\place(250,750)[\scriptstyle \alpha]

\efig
$$
we are given cubes

$$
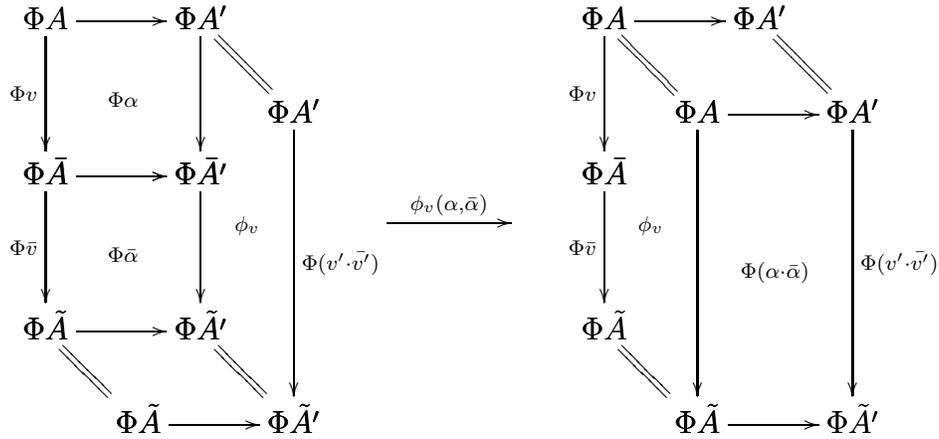
\bfig

\square(0,300)[\Phi\ov{A} `\Phi\ov{A'} `\Phi\ovv{A} `\Phi\ovv{A'};`\Phi\ov{v}``]

\square(0,800)[\Phi A `\Phi A'`\Phi \ov{A} `\Phi\ov{A'};`\Phi v``]

\square(300,0)/``>`>/<500,1000>[`\Phi A' `\Phi\ovv{A}`\Phi\ovv{A'};``\Phi (v' \cdot \ov{v'})`]

\morphism(0,300)/=/<300,-300>[\Phi\ovv{A}`\Phi\ovv{A};]

\morphism(500,300)/=/<300,-300>[\Phi\ovv{A'}`\Phi\ovv{A'};]

\morphism(500,1300)/=/<300,-300>[\Phi A'`\Phi A';]

\place(250,550)[\scriptstyle \Phi \ov{\alpha}]

\place(250,1050)[\scriptstyle \Phi \alpha]

\place(650,650)[\scriptstyle \phi_v]

\place(1300,650)[\to^{\phi_v(\alpha, \ov{\alpha})}]

\square(1800,800)/>`>``/[\Phi A `\Phi A'`\Phi\ov{A}`;`\Phi v``]

\morphism(1800,800)|l|<0,-500>[\Phi\ov{A}`\Phi\ovv{A};\Phi\ov{v}]

\square(2100,0)<500,1000>[\Phi A`\Phi A'`\Phi\ovv{A}`\Phi\ovv{A'};``\Phi (v'\cdot\ov{v'})`]

\morphism(1800,300)/=/<300,-300>[\Phi\ovv{A}`\Phi\ovv{A};]

\morphism(1800,1300)/=/<300,-300>[\Phi A`\Phi A;]

\morphism(2300,1300)/=/<300,-300>[\Phi A'`\Phi A';]

\place(2350,500)[\scriptstyle \Phi (\alpha \cdot \ov{\alpha})]

\place(1950,650)[\scriptstyle \phi_v]

\efig
$$
We denote this by
$$
\bfig

\place(0,0)[\phi_v : \frac{\Phi \alpha}{\Phi \ov{\alpha}} \ \to \ \Phi \left(\frac{\alpha}{\ov{\alpha}}\right)]

\place(1500,0)[(2)]

\place(-1500,0)[\ ]

\efig
$$

The transformation $ \underline{F} : (\id) F \to G(\id) $ is given by assigning to each object $ A $ of $ {\ic A} $ a horizontal arrow $ F(A) : \id_{FA} \to G(\id_A) $ of $ \ov{\mathbb B} $, i.e.\ to each object of $ {\ic A} $ a horizontal cell
$$
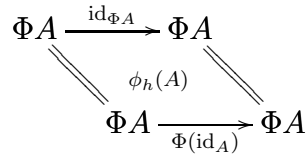
\bfig

\morphism(0,300)/=/<300,-300>[\Phi A `\Phi A;]

\morphism(0,300)|a|/>/<500,0>[\Phi A`\Phi A;\id_{\Phi A}]

\morphism(300,0)|b|/>/<500,0>[\Phi A`\Phi A;\Phi (\id_A)]

\morphism(500,300)/=/<300,-300>[\Phi A`\Phi A;]

\place(400,150)[\scriptstyle \phi_h(A)]

\efig
$$
and to each vertical arrow of $ {\mathbb A} $ a corresponding cell of $\ov{\mathbb B} $, i.e.\ for each vertical arrow $ A \todo{v} A' $ of $ {\ic A} $ a cube
$$
\bfig

\square(0,300)[\Phi A `\Phi A `\Phi A' `\Phi A'; \id_{\Phi A} `\Phi v `\Phi v `\id_{\Phi A'}]

\square(300,0)/``>`>/[` \Phi A `\Phi A' `\Phi A';` `\Phi v `\Phi(\id_{A'})]

\morphism(500,800)/=/<300,-300>[\Phi A `\Phi A;]

\morphism(0,300)|b|/=/<300,-300>[\Phi A'`\Phi A';]

\morphism(500,300)/=/<300,-300>[\Phi A' `\Phi A';]

\place(250,550)[\scriptstyle \id_{\Phi v}]

\place(400,150)[\scriptstyle \phi_h(A')]

\morphism(625,425)/=/<50,-50>[`;]

\place(1250,400)[\to^{\phi_h(v)}]

\square(1700,300)/>`>``/[\Phi A ` \Phi A `\Phi A' `;\id_{\Phi A} ` \Phi v ` `]

\square(2000,0)[\Phi A ` \Phi A `\Phi A' `\Phi A';  ` \Phi v `\Phi v `]

\morphism(1700,800)|b|/=/<300,-300>[\Phi A`\Phi A;]

\morphism(2200,800)/=/<300,-300>[\Phi A ` \Phi A;]

\morphism(1700,300)|b|/=/<300,-300>[\Phi A'`\Phi A';]

\place(2250,250)[\scriptstyle \Phi (\id_v)]

\morphism(1800,425)/=/<50,-50>[`;]

\place(2100,650)[\scriptstyle \phi_h(A')]

\efig
$$
We denote both of these by
$$
\bfig

\place(0,0)[\phi_h : \id_\Phi \to \Phi(\id)]

\place(1500,0)[(3)]

\place(-1500,0)[\ ]

\efig
$$
Similarly the transformation $ {\underline H} : mH \to Gm $ is given by
$$
\bfig

\place(0,0)[\phi_h : \Phi\alpha | \Phi \beta \to \Phi (\alpha|\beta).]

\place(1500,0)[(4)]

\place(-1500,0)[\ ]

\efig
$$

The transformations (1)-(4) have to satisfy the following conditions.

First, vertical laxity (laxity of $ F $ and $ G $)
$$
\bfig

\square/>`>`>`<-/<800,500>[\dfrac{\Id_{\Phi h}}{\Phi\alpha} 
`\dfrac{\Phi(\Id_h)}{\Phi\alpha} 
`\Phi\alpha 
`\Phi\left(\dfrac{\Id_h}{\alpha}\right);
\frac{\phi_v}{\Phi\alpha} 
`\lambda' 
`\phi_v 
`\Phi(\lambda')]

\place(400,250)[\scriptstyle (5)]

\square(1500,0)/>`>`>`<-/<800,500>[\dfrac{\Phi\alpha}{\Id_{\Phi \ov{h}}}  
   `\dfrac{\Phi \alpha}{\Phi (\Id_{\ov{h}})} 
   `\Phi \alpha 
   `\Phi \left(\dfrac{\alpha}{\Id_{\ov{h}}}\right);
\frac{\Phi\alpha}{\phi_v} 
`\rho' 
` \phi_v 
`\Phi(\rho')]

\place(1900,250)[\scriptstyle (6)]

\efig
$$

$$
\bfig

\square<1000,1000>[\dfrac{\Phi\alpha}{\dfrac{\Phi\beta}{\Phi\gamma}}
    `\dfrac{\Phi\left(\dfrac{\alpha}{\beta}\right)}{\Phi\gamma}
    `\dfrac{\Phi\alpha}{\Phi\left(\dfrac{\beta}{\gamma}\right)}
     `\Phi\left(\dfrac{\alpha}{\dfrac{\beta}{\gamma}}\right);
 \frac{\phi_v}{\Phi\gamma} `\frac{\Phi\alpha}{\phi_v} `\phi_v `\phi_v]

\place(500,500)[\scriptstyle (7)]

\efig
$$
Next, horizontal laxity (laxity of $ \underline{F}, \underline{H}$)
$$
\bfig

\square/>`>`>`<-/<800,500>[\id_{\Phi v} |\Phi\alpha 
         `\Phi(\id_v) | \Phi\alpha
     `\Phi\alpha 
` \Phi(\id_v | \alpha);
  \phi_h | \Phi\alpha `\lambda ` \phi_h `\Phi(\lambda)]

\place(400,250)[\scriptstyle (8)]

\square(1500,0)/>`>`>`<-/<800,500>[\Phi\alpha | \id_{\Phi w}
      `\Phi\alpha | \Phi(\id_w)
      `\Phi \alpha 
   `\Phi(\alpha | \id_w);
   \Phi\alpha | \phi_h `\rho ` \phi_h `\Phi(\rho)]

\place(1900,250)[\scriptstyle (9)]

\efig
$$
$$
\bfig

\square<1200,500>[\Phi \alpha | \Phi\beta| \Phi \gamma  `\Phi(\alpha |\beta) | \Phi \gamma 
    `\Phi \alpha |\Phi(\beta|\gamma)  ` \Phi (\alpha |\beta |\gamma);\phi_h |\Phi \gamma `\Phi \alpha |\phi_h `\phi_h `\phi_h]

\place(600,250)[\scriptstyle (10)]

\efig
$$
Finally, vertical/horizontal compatibility ($ {\underline F}, {\underline H} $ horizontal transformations)
$$
\bfig

\square/>`>``>/<600,500>[\Id_{\id_{\Phi A}}  
`\Id_{\Phi(\id_A)} 
`\id_{\Id_{\Phi A}} 
`\id_{\Phi(\Id_A)};
\Id_{\phi_h} 
`\tau 
`
`\id_{\phi_v}]

\square(600,0)/>``>`>/<600,500>[\Id_{\Phi (\id_A)}  
`\Phi(\Id_{\id_A}) 
`\id_{\Phi(\Id_A)} 
`\Phi(\id_{\Id_A});
\phi_v``\Phi(\tau)`\phi_h]

\place(600,250)[\scriptstyle (11)]

\efig
$$
$$
\bfig

\square/>`>``>/<800,500>[\dfrac{\id_{\Phi v}}{\id_{\Phi \ov{v}}}
   `\dfrac{\Phi\id_v}{\Phi\id_{\ov{v}}}
    `\id_{\frac{\Phi v}{\Phi \ov{v}}}
    `\id_{\Phi \frac{v}{\ov{v}}};
\frac{\phi_h}{\phi_h} `\mu ``\id_{\phi_v}]

\square(800,0)/>``>`>/<800,500>[\dfrac{\Phi\id_v}{\Phi\id_{\ov{v}}}
    `\Phi\left(\dfrac{\id_v}{\id_{\ov{v}}}\right)
    `\id_{\Phi \frac{v}{\ov{v}}}
    `\Phi(\id_{\frac{v}{\ov{v}}});
  \phi_v ``\Phi_\mu`\phi_h]

\place(800,250)[\scriptstyle (12)]

\efig
$$
$$
\bfig

\square/>`>``>/<800,500>[\Id_{\Phi h | \Phi k} 
`\Id_{\Phi (h | k)}
`\Id_{\Phi h} | \Id_{\Phi  k}
` \Phi(\Id_h) | \Phi(\Id_k);
\Id_{\phi_h} 
`\delta``\phi_v |\phi_v]

\square(800,0)/>``>`>/<800,500>[\Id_{\Phi (h | k)} 
` \Phi(\Id_{h | k}) 
`\Phi(\Id_h) | \Phi(\Id_k) 
`\Phi(\Id_h |\Id_k);
\phi_v ``\Phi\delta `\phi_h]

\place(800,250)[\scriptstyle (13)]

\efig
$$
$$
\bfig

\square/>`>``>/<1000,600>[\dfrac{\Phi\alpha | \Phi \beta}{\Phi \ov{\alpha} | \Phi \ov{\beta}}
     `\dfrac{\Phi(\alpha | \beta)}{\Phi(\ov{\alpha} | \ov{\beta})}
   `\left.\dfrac{\Phi\alpha}{\Phi\ov{\alpha}}\right|\dfrac{\Phi\beta}{\Phi\ov{\beta}}
 `\left.\Phi\left(\dfrac{\alpha}{\ov{\alpha}}\right)\right|\Phi \left(\dfrac{\beta}{\ov{\beta}}\right);
  \frac{\phi_h}{\phi_h}
`\chi ``\phi_v | \phi_v]

\square(1000,0)/>``>`>/<1000,600>[\dfrac{\Phi(\alpha | \beta)}{\Phi(\ov{\alpha} | \ov{\beta})}
    `\Phi\left(\dfrac{\alpha | \beta}{\ov{\alpha} | \ov{\beta}} \right)
  `\left.\Phi\left(\dfrac{\alpha}{\ov{\alpha}}\right)\right|\Phi \left(\dfrac{\beta}{\ov{\beta}}\right)
  `\Phi\left(\left.\dfrac{\alpha}{\ov{\alpha}}\right|\dfrac{\beta}{\ov{\beta}}\right);
   \phi_v ``\Phi(\chi) `\phi_h]

\place(1000,300)[\scriptstyle (14)]

\efig
$$

The data (1)-(4) satisfying conditions (5)-(14) is a complete description of what we call {\em lax-lax morphisms} of intercategories. They compose in the obvious way and the unit laws and associativity hold strictly. This is because composition comes from the transversal structure.

If, instead, we consider colax morphisms, still between pseudocategories in $ \Dlax $, then the $ F, G, H $ above are still lax functors, but now the transformations $ {\underline F} $ and ${\underline H} $ go in the opposite direction. This means that the transformations $ \phi_v $ of (1) and (2) remain the same whereas the $ \phi_h $ of (3) and (4) are in the opposite direction. Examining conditions (5)-(14) we see that (5)-(7) are unchanged but in (8)-(14) the reversal of $ \phi_h $ produces a new commutative diagram of a similar sort. We call this sort of morphism {\em colax-lax}.

We can now give an explicit description of the cells $ \pi $ introduced in Theorem \ref{PC}. Consider a diagram of intercategories
$$
\bfig

\square[{\ic A} `{\ic B} `{\ic C} `{\ic D};\Phi ` \Sigma ` \Theta`\Psi]

\place(250,250)[\scriptstyle \pi]

\efig
$$
where $ \Phi $ and $ \Psi $ are lax-lax and $ \Sigma $ and $ \Theta $ colax-lax. A cell $ \pi $ as above consists of:

\noindent -- for every object $ A $ of $ {\ic A} $, a transversal arrow in $ {\ic D} $
$$\bfig

\place(-1500,0)[\ ]

\place(0,0)[\pi A : \Theta \Phi A \to \Psi \Sigma A]

\place(1500,0)[(1)]

\efig
$$ 
-- for every vertical arrow $ v : A \tod \ov{A} $ of $ {\ic A} $, a vertical cell in $ {\ic D} $
$$
\bfig
\btriangle(0,500)/@{>}|{\bb}`>`/[\Theta \Phi A `\Theta \Phi \ov{A} `\Psi \Sigma A;\Theta \Phi v `\pi A` \pi v]

\qtriangle(0,0)/`>`@{>}|{\bb}/[\Theta \Phi \ov{A} ` `\Psi \Sigma\ov{A};`\pi\ov{A}`\Psi\Sigma v]

\efig
$$
$$\bfig

\place(-1500,0)[\ ]

\place(0,0)[\pi v : \Theta \Phi v \to \Psi \Sigma v ]

\place(1500,0)[(2)]

\efig
$$ 
-- for every horizontal arrow $ f : A \toc A' $ of $ {\ic A} $ a horizontal cell in $ {\ic D} $
$$
\bfig\scalefactor{1.2}

\morphism(0,0)/@{>}|{\cc}/[\Theta \Phi A `\Theta \Phi A';\Theta \Phi f]

\morphism(300,-300)|b|/@{>}|{\cc}/[\Psi \Sigma A `\Psi \Sigma A';\Psi \Sigma f]

\morphism(0,0)|b|<300,-300>[\Theta \Phi A`\Psi \Sigma A;\pi A]

\morphism(500,0)<300,-300>[\Theta \Phi A'`\Psi \Sigma A';\pi A']

\place(400,-150)[\scriptstyle \pi f]

\efig
$$
$$\bfig

\place(-1500,0)[\ ]

\place(0,0)[\pi f : \Theta \Phi f \to \Psi \Sigma f]

\place(1500,0)[(3)]

\efig
$$
-- for every basic cell
$$
\bfig

\square/>`@{>}|{\bb}`@{>}|{\bb}`>/[A `A'`\ov{A} `\ov{A'}; f `v `v'`\ov{f}]

\place(250,250)[\scriptstyle \alpha]

\efig
$$
a cube $ \pi \alpha $ in ${\ic D} $
$$
\bfig\scalefactor{1.2}

\square(0,300)/@{>}|{\cc}`@{>}|{\bb}``/[\Theta \Phi A ` \Theta \Phi A' `\Theta\Phi\ov{A}`;\Theta\Phi f `\Theta\Phi v` `]

\square(300,0)/@{>}|{\cc}`@{>}|{\bb}`@{>}|{\bb}`@{>}|{\cc}/[\Psi \Sigma A ` \Psi \Sigma A' `\Psi \Sigma \ov{A} `\Psi \Sigma \ov{A'};``\Psi \Sigma \ov{v}`\Psi\Sigma \ov{f}]

\morphism(0,800)|b|<300,-300>[\Theta\Phi A`\Psi\Sigma A;]

\morphism(500,800)<300,-300>[\Theta\Phi A' ` \Psi\Sigma A';\pi A']

\morphism(0,300)|b|<300,-300>[\Theta \Phi \ov{A}`\Psi \Sigma \ov{A};\pi\ov{A}]

\place(550,250)[\scriptstyle \Psi \Sigma \alpha]

\place(150,400)[\scriptstyle \pi v]

\place(400,650)[\scriptstyle \pi f]

\efig
$$
$$
\bfig

\place(0,0)[\pi \alpha : \Theta \Phi \alpha \to \Psi \Sigma \alpha]

\place(-1500,0)[\ ]

\place(1500,-0)[(4)] 

\efig
$$

(1)-(4) set up the structure of $ \pi $ and make the domains and codomains explicit, which are just as one would expect. We give the equations that $ \pi $ must satisfy only at the highest level with the understanding that they imply corresponding ones lower down so as to make the domains and codomains work.

There are two groups of equations. The first expressing that $ \pi $ is made up of $2$-cells in $ \Dlax $.
$$
\bfig

\square/>`>``>/<600,500>[\Id_{\Theta\Phi h} 
`\Theta(\Id_{\Phi h}) 
`\Id_{\Psi\Sigma h} 
`\Psi(\Id_{\Sigma h});
\theta_v 
`\Id_{\pi h}``\psi_v]

\square(600,0)/>``>`>/<600,500>[\Theta(\Id_{\Phi h}) 
`\Theta \Phi(\Id_h) 
`\Psi(\Id_{\Sigma h}) 
`\Psi\Sigma(\Id_h);
\Theta \phi_v ``\pi(\Id_h) `\Psi(\sigma_v)]

\place(600,250)[\scriptstyle (5)]

\efig
$$

$$
\bfig

\square/>`>``>/<600,500>[\frac{\Theta\Phi\alpha}{\Theta\Phi\beta}  `\Theta(\frac{\Phi\alpha}{\Phi\beta}) `\frac{\Psi\Sigma\alpha}{\Psi\Sigma\beta} ` \Psi(\frac{\Sigma\alpha}{\Sigma\beta});\theta_v `\frac{\pi\alpha}{\pi\beta} ``\psi_v]

\square(600,0)/>``>`>/<600,500>[\Theta(\frac{\Phi\alpha}{\Phi\beta}) `\Theta\Phi(\frac{\alpha}{\beta}) `\Psi(\frac{\Sigma\alpha}{\Sigma\beta}) `\Psi\Sigma(\frac{\alpha}{\beta});\Theta\phi_v``\pi(\frac{\alpha}{\beta})`\Psi(\sigma_v)]

\place(600,250)[\scriptstyle (6)]

\efig
$$
The second group expressing that $ \pi $ is a cell as in Theorem \ref{PC}.
$$
\bfig

\Ctriangle/<-``>/[\id_{\Theta\Phi v}
`\Theta(\id_{\Phi v}) 
`\Theta\Phi(\id_v);
\theta_h``\Theta\phi_h]

\square(500,0)/>```>/<700,1000>[\id_{\Theta\Phi v} 
` \id_{\Psi\Sigma v}
` \Theta\Phi(\id_v) 
`\Psi\Sigma(\id_v);
\id_{\pi v} ```\pi(\id_v)]

\Dtriangle(1200,0)/`>`<-/[\id_{\Psi\Sigma v} 
`\Psi(\id_{\Sigma v})
`\Psi\Sigma(\id_v);
`\psi_h`\Psi(\sigma_h)]

\place(900,500)[\scriptstyle (7)]

\efig
$$

$$
\bfig

\Ctriangle/<-``>/[\Theta\Phi\alpha|\Theta\Phi\beta `\Theta(\Phi\alpha | \Phi \beta) `\Theta\Phi (\alpha| \beta); \theta_h``\Theta(\phi_h)]

\square(500,0)/>```>/<900,1000>[\Theta\Phi\alpha|\Theta\Phi\beta `\Psi\Sigma\alpha | \Psi \Sigma\beta `\Theta\Phi (\alpha| \beta) `\Psi\Sigma(\alpha|\beta);\pi\alpha | \pi\beta ```\pi(\alpha|\beta)]

\Dtriangle(1400,0)/`>`<-/[\Psi\Sigma\alpha | \Psi \Sigma\beta `\Psi(\Sigma\alpha|\Sigma\beta)`\Psi\Sigma(\alpha|\beta);`\psi_h`\Psi(\sigma_h)]

\place(1000,500)[\scriptstyle (8)]

\efig
$$

Horizontal and vertical composition comes from the pasting of cubes, as explained in Theorem \ref{PC}, and is defined in terms of transversal composition. For cells
$$
\bfig

\square[{\ic B} `{\ic B'} `{\ic D} `{\ic D'};\Phi' `\Theta `\Theta'`\Psi']

\place(850,250)[\mbox{and}]

\square(1200,0)[{\ic C} `{\ic D} `\ov{\ic C} ` \ov{\ic D};\Psi `\ov{\Sigma} `\ov{\Theta} `\ov{\Psi}]

\place(250,250)[\scriptstyle \pi']

\place(1450,250)[\scriptstyle \ov{\pi}]

\efig
$$

$$
(\pi | \pi') (\alpha)  =  (\Theta' \Phi' \Phi \alpha \to^{\pi' \Phi \alpha} \Psi' \Theta \Phi \alpha \to^{\Psi'\pi\alpha} \Psi'\Psi \Sigma \alpha) = (\Psi'\pi\alpha) (\pi'\Phi \alpha),
$$
$$
\frac{\pi}{\ov{\pi}} (\alpha) = (\ov{\Theta} \Theta \Phi \alpha \to^{\ov{\Theta}\pi\alpha} \ov{\Theta}\Psi\Sigma\alpha \to^{\ov{\pi}\Sigma\alpha} \ov{\Psi}\ov{\Sigma} \Sigma \alpha) = (\ov{\pi}\Sigma \alpha) (\ov{\Theta} \pi \alpha).
$$

We have just described the double category $ {\mathbb P}{\rm s}{\mathbb C}{\rm at}(\Dlax) $. We can now do the same for $ {\mathbb P}{\rm s}{\mathbb C}{\rm at}(\Dcolax) $ keeping in mind the identification of pseudocategories in $ \Dlax $ with pseudocategories in $ \Dcolax $.

An internal lax functor is, in part, a diagram
$$
\bfig

\square/>`>`>`>/[{\mathbb X}_2 ` {\mathbb X}_1 `{\mathbb Y}_2 `{\mathbb Y}_1;`T_2``]

\square|allb|/@{>}@<-3pt>`` `@{>}@<-3pt>/[{\mathbb X}_2 ` {\mathbb X}_1 `{\mathbb Y}_2 `{\mathbb Y}_1;```]

\square|allb|/@{>}@<3pt>```@{>}@<3pt>/[{\mathbb X}_2 ` {\mathbb X}_1 `{\mathbb Y}_2 `{\mathbb Y}_1;```]

\square(500,0)/<-`>`>`<-/[{\mathbb X}_1 ` {\mathbb X}_0 `{\mathbb Y}_1 `{\mathbb Y}_0;`T_1`T_0`]

\square(500,0)|allb|/@{>}@<-3pt>```@{>}@<-3pt>/[{\mathbb X}_1 ` {\mathbb X}_0 `{\mathbb Y}_1 `{\mathbb Y}_0;```]

\square(500,0)|allb|/@{>}@<3pt>```@{>}@<3pt>/[{\mathbb X}_1 ` {\mathbb X}_0 `{\mathbb Y}_1 `{\mathbb Y}_0;```]

\efig
$$
where the $ T_i $ are colax functors of double categories. So $ T_0 $, for example, comes with comparison cells
$$
T_0 (x \cdot \ov{x}) \to T_0 x \cdot T_0\ov{x}, \quad  T_0(\id) \to \id_{T_0},
$$
where $ x $ and $ \ov{x} $ are vertical arrows of $ {\mathbb X}_0 $. Now, when an intercategory $ {\ic A} $ is considered as a pseudo-category in $ \Dcolax $ as above, the vertical arrows of $ {\mathbb X}_0 $ are the horizontal arrows of $ {\ic A} $. It will follow then that if a morphism of intercategories is represented as a diagram in $ \Dcolax $ it will be horizontally colax. If we further have cells
$$
\bfig

\square[{\mathbb X}_0 `{\mathbb X}_1 ` {\mathbb Y}_0 `{\mathbb Y}_1;\Id ` T_0 `T_1 `\Id]

\morphism(200,200)/=>/<100,100>[`;]

\square(1000,0)[{\mathbb X}_2 ` {\mathbb X}_1 `{\mathbb Y}_2 `{\mathbb Y}_1;M `T_2 `T_1`M]

\morphism(1200,200)/=>/<100,100>[`;]

\efig
$$
to give us an internal lax functor, we get comparisons for vertical identities and composition at the intercategory level. That is, a lax functor of pseudocategories in $ \Dcolax $ gives a colax-lax morphism of intercategories. This is the same as a colax functor of pseudocategories in $ \Dlax $.

We summarize this discussion in the following.

\begin{theorem} \label{CLM} Under the identifications of Theorem \ref{equivalence}, colax-lax morphisms correspond to either internal colax functors between pseudocategories in $ \Dlax $ or internal lax functors between pseudocategories in $ \Dcolax $.

\end{theorem}

On the other hand, if we consider colax functors between pseudocategories in $ \Dcolax $ we get something new, a colax-colax functor of intercategories.

We also have cells
$$
\bfig

\square[{\ic A} ` {\ic B} ` {\ic C} `{\ic D}; \Phi ` \Sigma ` \Theta ` \Psi]

\place(250,250)[\scriptstyle \pi]

\efig
$$
where $ \Phi $ and $ \Psi $ are colax-lax and $ \Sigma $ and $ \Theta $ are colax-colax. For a basic cell $ \alpha $ in ${\ic A} $
$$
\pi \alpha : \Theta \Phi \alpha \to \Psi \Sigma \alpha.
$$

{\sc Remark}: The notion of a lax-colax morphism of intercategories (horizontally lax and vertically colax) doesn't come up and in fact the vertical/horizontal compatibility conditions (11)-(14) don't make sense in this case. For example, (14) would look like
$$
\bfig

\square/>`>``<-/<1000,500>[\dfrac{\Phi \alpha|\Phi \beta}{\Phi \gamma|\Phi \delta} 
  `\dfrac{\Phi(\alpha | \beta)}{\Phi(\gamma|\delta)} 
  `\left.\dfrac{\Phi \alpha}{\Phi \gamma}\right|\dfrac{\Phi\beta}{\Phi \delta} 
   `\Phi\left(\left.\dfrac{\alpha}{\gamma}\right)\right|\Phi\left(\dfrac{\beta}{\delta}\right);
\frac{\phi_h}{\phi_h}`\chi``\phi_v|\phi_v]

\square(1000,0)/<-``>`>/<1000,500>[\dfrac{\Phi(\alpha | \beta)}{\Phi(\gamma|\delta)} 
   `\Phi\left(\dfrac{\alpha|\beta}{\gamma | \delta}\right)
  `\Phi\left(\left.\dfrac{\alpha}{\gamma}\right)\right|\Phi\left(\dfrac{\beta}{\delta}\right)
   `\Phi\left(\left.\dfrac{\alpha}{\gamma}\right|\dfrac{\beta}{\delta}\right);
\phi_v ``\Phi (\chi)`\phi_h]

\efig
$$

To sum up what we have so far, there are three kinds of morphisms of intercategories, lax-lax, colax-lax, and colax-colax and there are two kinds of cells relating the lax-lax with colax-lax and the colax-lax with colax-colax. These cells give us two strict double categories of intercategories. Furthermore, these double categories have the colax-lax morphisms in common. This suggests that we might have an intercategory of intercategories with the colax-lax morphisms as transversal. In fact it is much better: we have a strict triple category.

In order to complete the construction we first have to define double cells relating lax-lax functors with colax-colax ones. Let $ \Phi $ and $ \Psi $ be lax-lax and $ \Sigma $ and $ \Theta $ colax-colax. A double cell
$$
\bfig

\square[{\ic A} ` {\ic B} `{\ic C} `{\ic D};\Phi `\Sigma `\Theta`\Psi]

\place(250,250)[\scriptstyle \pi]

\efig
$$
consists of the same data (1)-(4) as above, satisfying the conditions
$$
\bfig

\Ctriangle/<-``>/[\Id_{\Theta\Phi}`\Theta(\Id_\Phi) `\Theta\Phi(\Id);\theta_v``\Theta\phi_v]

\square(500,0)/>```>/<700,1000>[\Id_{\Theta\Phi} ` \Id_{\Psi\Sigma}` \Theta\Phi(\Id) `\Psi\Sigma(\Id);\Id_\pi ```\pi(\Id)]

\Dtriangle(1200,0)/`>`<-/[\Id_{\Psi\Sigma} `\Psi(\Id_\Sigma)`\Psi\Sigma(\Id);`\Psi_v`\Psi(\sigma_v)]

\place(900,500)[\scriptstyle (5')]

\efig
$$

$$
\bfig

\Ctriangle/<-``>/[\dfrac{\Theta \Phi \alpha}{\Theta \Phi \beta}
     `\Theta\left(\dfrac{\Phi\alpha}{\Phi\beta}\right)
     `\Theta \Phi \left(\dfrac{\alpha}{\beta}\right);
 \theta_v``\Theta\phi_v]

\square(500,0)/>```>/<1000,1000>[\dfrac{\Theta \Phi \alpha}{\Theta \Phi \beta}
    `\dfrac{\Psi\Sigma \alpha}{\Psi\Sigma \beta}
   `\Theta \Phi \left(\dfrac{\alpha}{\beta}\right)
    `\Psi \Sigma \left(\dfrac{\alpha}{\beta}\right);
   \frac{\pi \alpha}{\pi \beta}
 ```\pi\left(\frac{\alpha}{\beta}\right)]

\Dtriangle(1500,0)/`>`<-/[\dfrac{\Psi\Sigma \alpha}{\Psi\Sigma \beta}
   `\Psi\left(\dfrac{\Sigma \alpha}{\Sigma \beta}\right)
   `\Psi \Sigma \left(\dfrac{\alpha}{\beta}\right);
  `\psi_v `\Psi (\sigma_v)]

\place(1000,500)[\scriptstyle (6')]

\efig
$$
as well as conditions (7) and (8) above.

Such cells compose horizontally and vertically in the same way as the cells for the lax-lax, colax-lax case. The data is the same, the conditions (7) and (8) are the same and (5') and (6') are the same as the corresponding ones in the colax-lax, colax-colax case. In fact the conditions (5) and (6) and the conditions (7) and (8) are independent, the former referring only to the vertical structure and the latter to the horizontal. This gives us a strict double category which will comprise the basic cells of our triple category.

\begin{theorem}\label{mainth} There is a strict triple category $ {\ic ICat} $ whose objects are intercategories, with colax-lax morphisms as transversal arrows, lax-lax morphisms as horizontal arrows, colax-colax morphisms as vertical arrows, with the three kinds of double cells defined above and with commuting cubes as triple cells.

\end{theorem}

\begin{proof} We must say what is meant by a commuting cube. Consider a cube of cells with back and front as shown
$$
\bfig

\square(0,300)[{\ic A} `{\ic B} `{\ic C} `{\ic D}; \Phi `\Sigma `\Theta `\Psi]

\square(300,0)/``>`>/[`{\ic B'}`{\ic C'}`{\ic D'};` `\Theta' `\Psi']

\morphism(500,800)<300,-300>[{\ic B} `{\ic B'};K']

\morphism(0,300)|b|<300,-300>[{\ic C}`{\ic C'};L]

\morphism(500,300)<300,-300>[{\ic D} `{\ic D'};L']

\place(250,550)[\scriptstyle \pi]

\place(400,150)[\scriptstyle \ov{\kappa}]

\place(650,400)[\scriptstyle \ov{\lambda}]

\square(1400,300)/>`>``/[{\ic A} ` {\ic B} `{\ic C} `;\Phi`\Sigma ` `]

\square(1700,0)[{\ic A'} ` {\ic B'} `{\ic C'} `{\ic D'}; \Phi' ` \Sigma' `\Theta' `\Psi']

\morphism(1400,800)|b|/>/<300,-300>[{\ic A}`{\ic A'};K]

\morphism(1900,800)/>/<300,-300>[{\ic B} ` {\ic B'};K']

\morphism(1400,300)|b|/>/<300,-300>[{\ic C}`{\ic C'};L]

\place(1950,250)[\scriptstyle \pi']

\place(1550,400)[\scriptstyle \lambda]

\place(1800,650)[\scriptstyle \kappa]

\efig
$$
The data for the cells is given by
$$
\pi \alpha : \Theta \Phi \alpha \to \Psi \Sigma \alpha,
$$
$$
\pi' \alpha' : \Theta' \Phi' \alpha' \to \Psi' \Sigma' \alpha',
$$
$$
\kappa \alpha : K' \Phi \alpha \to \Phi' K \alpha,
$$
$$
\ov{\kappa} \gamma : L' \Psi \gamma \to \Psi' L \gamma,
$$
$$
\lambda \alpha : \Sigma' K \alpha \to L \Sigma \alpha,
$$
$$
\ov{\lambda} \beta : \Theta' K' \beta \to L' \Theta \beta.
$$
Commutativity means that for every $ \alpha $
$$
\bfig

\Ctriangle/<-``>/[\Theta' \Phi' K \alpha`\Theta' K' \Phi \alpha`L' \Theta \Phi \alpha;\Theta \kappa \alpha``\ov{\lambda} \Phi\alpha]

\square(500,0)/>```>/<700,1000>[\Theta' \Phi' K \alpha` \Psi' \Sigma' K \alpha` L' \Theta \Phi \alpha`L' \Psi \Sigma\alpha;\pi' K \alpha ```L'\pi\alpha]

\Dtriangle(1200,0)/`>`<-/[ \Psi' \Sigma' K \alpha`\Psi' L \Sigma \alpha`L' \Psi \Sigma\alpha;`\Psi'\lambda \alpha`\ov{\kappa}\Sigma\alpha]

\efig
$$
commutes. That commuting cubes paste in all three directions is a straightforward, if long, calculation which we omit.

\end{proof}


\section{Pseudocategories in a double category}\label{mystery}

Theorem \ref{equivalence} asserts that double pseudocategories in $ \CAT $, horizontal intercategories and vertical intercategories are the same but it later turns out that the morphisms are not. Double pseudocategories have three variants of morphism, lax-lax, colax-lax and colax-colax whereas horizontal (resp.\ vertical) only have two variants of morphism, lax-lax and colax-lax (resp.\ colax-lax and colax-colax). We should then conclude that categorically the structures are not the same: if they were we should be able to recover the third variant. If we concentrate on the presentation in $ \Dlax $, the reason for this discrepancy is that this $2$-category is only the horizontal part of $ \Doub $, while the third kind of morphism is a double category construct.

We want thus to define a horizontal intercategory {\em in the double category} $\Doub $. First, we need one more concept.

\begin{definition} A cell
$$
\bfig

\square/>`@{>}|{\bb}`@{>}|{\bb}`>/[{\mathbb A} `{\mathbb B} `{\mathbb C} `{\mathbb D};F `U `V`G]

\place(250,250)[\scriptstyle \sigma]

\efig
$$
in $ \Doub $ is ({\em horizontally}) {\em strict} if $ F $ and $ G $ are strict functors and $\sigma $ is a horizontal identity on objects and vertical arrows of $ {\mathbb A} $, i.e.\ $VF = GU$.

\end{definition}

\begin{proposition} The horizontal (resp.\ vertical) composite of strict cells is again strict.

\end{proposition}

\begin{proposition} Let
$$
\bfig

\square/>`@{>}|{\bb}`@{>}|{\bb}`>/[{\mathbb A} `{\mathbb B} `{\mathbb C} `{\mathbb D};F `U `V `G]

\place(250,250)[\scriptstyle \sigma]

\place(900,250)[\mbox{\ and\ }]

\square(1300,0)/>`@{>}|{\bb}`@{>}|{\bb}`>/[{\mathbb A'} `{\mathbb B} `{\mathbb C'} `{\mathbb D'};F' `U' `V `G']

\place(1550,250)[\scriptstyle \sigma']

\efig
$$
be strict cells in $ \Doub $. Then there exist unique $ U \times_V U' : {\mathbb A} \times_{\mathbb B} {\mathbb A'} \tod {\mathbb C} \times_{\mathbb D} {\mathbb C'} $ and strict cells
$$
\bfig

\square/>`@{>}|{\bb}`@{>}|{\bb}`>/[{\mathbb A} \times_{\mathbb B} {\mathbb A'}  `{\mathbb A}
   `{\mathbb C} \times_{\mathbb D}{\mathbb C'}   `{\mathbb C};
    p_1  `U \times_V U' `U `q_1]

\place(250,250)[\scriptstyle \pi_1]

\place(900,250)[\mbox{\ and\ }]

\square(1500,0)/>`@{>}|{\bb}`@{>}|{\bb}`>/[{\mathbb A} \times_{\mathbb B} {\mathbb A'}  `{\mathbb A'}
   `{\mathbb C} \times_{\mathbb D}{\mathbb C'}   `{\mathbb C'};
    p_2  `U \times_V U' `U' `q_2]

\place(1750,250)[\scriptstyle \pi_2]

\efig
$$

Further
$$
\bfig

\square[U \times_V U' `U' `U `V;\pi' `\pi `\sigma' `\sigma]

\efig
$$
is a pullback in $ {\bf Dbl}_1 $, the category whose objects are colax functors and whose morphisms are arbitrary cells in $ \Doub $.

\end{proposition}

\begin{proof} (Sketch) As $ \pi_1 $ and $ \pi_2 $ are to be strict, $ U \times_V U' $ has to be given by
$$
U \times_V U' (A, A') = (UA, U'A')
$$
with similar definitions for arrows, cells, and the colaxity structural cells, i.e.~everything happens at the set-theoretical level. The universal property is also like that: it is simply a question of writing it out.

\end{proof}

Let $ SC(\Doub) $ be the double category of strict horizontal cospans in $ \Doub $. The objects are cospans
$$
{\mathbb A} \to^{F} {\mathbb B} \to/<-/^{\ \ F'} {\mathbb A'}
$$
of strict functors (considered as horizontal morphisms of $ \Doub $). The horizontal morphisms are commutative diagrams
$$
\bfig

\square[{\mathbb A}_1  `{\mathbb B}_1 `{\mathbb A}_2 `{\mathbb B}_2;F_1 `G `H `F'_2]

\square(500,0)/<-``>`<-/[{\mathbb B}_1 `{\mathbb A'}_1 `{\mathbb B}_2 `{\mathbb A'}_2;F'_1 ``G'`F'_2]

\efig
$$
of horizontal morphisms of $ \Doub $, i.e.~$G, H, G' $ are lax functors. The vertical morphisms are strict cells
$$
\bfig

\square/>`@{>}|{\bb}`@{>}|{\bb}`>/[{\mathbb A} `{\mathbb B}`{\mathbb C}`{\mathbb D};F`U`V`G]

\place(250,250)[\scriptstyle \sigma]

\square(500,0)/<-`@{>}|{\bb}`@{>}|{\bb}`<-/[{\mathbb B}`{\mathbb A'} `{\mathbb D}`{\mathbb C'};F'``U'`G']

\place(750,250)[\scriptstyle \sigma']

\efig
$$
i.e.~$U, V, U' $ are colax functors and the diagram commutes. A cell consists of three arbitrary cells
$$
\bfig

\square/>`@{>}|{\bb}`@{>}|{\bb}`>/[{\mathbb A}_1 `{\mathbb A}_2 `{\mathbb C}_1 `{\mathbb C}_2;G`U_1`U_2`K]

\place(250,250)[\scriptstyle \alpha]

\square(1000,0)/>`@{>}|{\bb}`@{>}|{\bb}`>/[{\mathbb B}_1 `{\mathbb B}_2 `{\mathbb D}_1 `{\mathbb D}_2;H`V_1`V_2`L]

\place(1250,250)[\scriptstyle \beta]

\square(2000,0)/>`@{>}|{\bb}`@{>}|{\bb}`>/[{\mathbb A'}_1  `{\mathbb A'}_2  `{\mathbb C'}_1  `{\mathbb C'}_2;G'`U'_1 `U'_2 `K']

\place(2250,250)[\scriptstyle \alpha']

\efig
$$
making
$$
\bfig

\square(0,300)/>`@{>}|{\bb}``/[{\mathbb A}_1 `{\mathbb B}_1`{\mathbb C}_1`;`U_1``]

\square(300,0)|arrb|/>`@{>}|{\bb}`@{>}|{\bb}`>/[{\mathbb A}_2 `{\mathbb B}_2 `{\mathbb C}_2 `{\mathbb D}_2;`U_2`V_2`]

\morphism(0,800)|b|<300,-300>[{\mathbb A}_1`{\mathbb A}_2;]

\morphism(500,800)<300,-300>[{\mathbb B}_1`{\mathbb B}_2;]

\morphism(0,300)|b|<300,-300>[{\mathbb C}_1`{\mathbb C}_2;]

\square(800,0)/<-`@{>}|{\bb}`@{>}|{\bb}`<-/[{\mathbb B}_2`{\mathbb A'}_2`{\mathbb D}_2`{\mathbb C'}_2;``U'_2`]

\morphism(1000,800)/>/<300,-300>[{\mathbb A'}_1`{\mathbb A'}_2;]

\morphism(970,790)/-/<0,-250>[`;]

\place(300,700)[\scriptstyle \sigma_1]

\place(800,700)[\scriptstyle \sigma'_1]

\place(600,600)[\scriptstyle \beta]

\place(1100,600)[\scriptstyle \alpha']

\place(150,400)[\scriptstyle \alpha]

\place(570,250)[\scriptstyle \sigma_2]

\place(1100,250)[\scriptstyle \sigma'_2]

\morphism(500,800)/>/<300,-300>[{\mathbb B}_1`{\mathbb B}_2;]

\morphism(470,790)/-/<0,-250>[`;]

\morphism(500,800)/<-/<500,0>[{\mathbb B}_1`{\mathbb A'}_1;]




\efig
$$
commute, i.e.~$\sigma_2 \alpha = \beta \sigma_1 $ and $ \sigma'_2 \alpha' = \beta \sigma'_1$.

The following result follows immediately from the description given above. We use the same convention regarding subscripts as in Section \ref{pscat}.

\begin{proposition} Pullback extends to a strict functor of double categories $ SC(\Doub) \to \  \Doub $.

\end{proposition}

\begin{definition} Let $ {\ic A} = {\mathbb A}_2\  \threepppp/>`>`>/<400>^{} |{} _{} \ {\mathbb A}_1 \ \three/>`<-`>/^{} | {}_{} \ {\mathbb A}_0 $ and $ {\ic B} = {\mathbb B}_2\  \threepppp/>`>`>/<400>^{} |{} _{} \ {\mathbb B}_1 \ \three/>`<-`>/^{} | {}_{} \ {\mathbb B}_0 $ be horizontal intercategories. A {\em vertical morphism} $ {\ic U} : {\ic A} \to {\ic B} $ consists of

\noindent (1) vertical arrows (colax functors) $ U_i : {\mathbb A}_i \tod {\mathbb B}_i $

\noindent (2) strict cells
$$
\bfig

\square/>`@{>}|{\bb}`@{>}|{\bb}`>/[{\mathbb A}_1 `{\mathbb A}_0 `{\mathbb B}_1 `{\mathbb B}_0;\partial_i `U_1 `U_0 `\partial_i]

\place(250,250)[\scriptstyle \delta_i]

\square(1000,0)/>`@{>}|{\bb}`@{>}|{\bb}`>/[{\mathbb A}_2 `{\mathbb A}_1 `{\mathbb B}_2 `{\mathbb B}_1;p_i `U_2 `U_1 `p_i]

\place(1250,250)[\scriptstyle \pi_i]

\efig
$$
and

\noindent (3) arbitrary cells
$$
\bfig

\square/>`@{>}|{\bb}`@{>}|{\bb}`>/[{\mathbb A}_0 `{\mathbb A}_1 `{\mathbb B}_0 `{\mathbb B}_1;\id `U_0`U_1`\id]

\place(250,250)[\scriptstyle \eta]

\square(1000,0)/>`@{>}|{\bb}`@{>}|{\bb}`>/[{\mathbb A}_2 `{\mathbb A}_1 `{\mathbb B}_2 `{\mathbb B}_1;m`U_2`U_1`m]

\place(1250,250)[\scriptstyle \mu]

\efig
$$
satisfying

$$
\bfig\scalefactor{.8}
\Vtriangle/`>`<-/<500,300>[{\mathbb B}_3 `{\mathbb B}_1`{\mathbb B}_2;`m_{12}`m]
\Vtriangle(0,700)/`>`<-/<500,300>[{\mathbb A}_3`{\mathbb A}_1`{\mathbb A}_2;`m_{12}`m]
\Atriangle(0,1000)/<-`>`/<500,300>[{\mathbb A}_2`{\mathbb A}_3`{\mathbb A}_1;m_{23}`m`]
\morphism(0,1000)|l|/@{>}|{\bb}/<0,-700>[{\mathbb A}_3`{\mathbb B}_3;U_3]
\morphism(500,700)|r|/@{>}|{\bb}/<0,-700>[{\mathbb A}_2`{\mathbb B}_2;U_2]
\morphism(1000,1000)|r|/@{>}|{\bb}/<0,-700>[{\mathbb A}_1`{\mathbb B}_1;U_1]

\place(1500,650)[=]

\Vtriangle(2000,0)/`>`<-/<500,300>[{\mathbb B}_3 `{\mathbb B}_1`{\mathbb B}_2;``]
\Atriangle(2000,300)/<-`>`/<500,300>[{\mathbb B}_2`{\mathbb B}_3`{\mathbb B}_1;``]
\Atriangle(2000,1000)/<-`>`/<500,300>[{\mathbb A}_2`{\mathbb A}_3`{\mathbb A}_1;m_{23}`m`]
\morphism(2000,1000)|l|/@{>}|{\bb}/<0,-700>[{\mathbb A}_3`{\mathbb B}_3;U_3]
\morphism(2500,1300)|r|/@{>}|{\bb}/<0,-700>[{\mathbb A}_2`{\mathbb B}_2;U_2]
\morphism(3000,1000)|r|/@{>}|{\bb}/<0,-700>[{\mathbb A}_1`{\mathbb B}_1;U_1]

\place(250,500)[\scriptstyle \mu_{12}]
\place(750,500)[\scriptstyle \mu]
\morphism(500,1150)|r|/=>/<0,-200>[`;\kappa]

\place(2250,800)[\scriptstyle \mu_{23}]
\place(2750,800)[\scriptstyle \mu]
\morphism(2500,400)|r|/=>/<0,-200>[`;\kappa]

\place(-700,650)[\ \ ]
\place(3700,650)[(4)]

\efig
$$

$$
\bfig\scalefactor{.8}

\Atriangle/<-`>`=/<500,350>[{\mathbb B}_2 `{\mathbb B}_1 `{\mathbb B}_1;\id_1`m`]

\morphism(500,250)|r|/=>/<0,-175>[`;\scriptstyle \lambda]

\Atriangle(0,700)/<-`>`/<500,350>[{\mathbb A}_2`{\mathbb A}_1`{\mathbb A}_1;\id_1`m`]

\morphism(0,700)|l|/@{>}|{\bb}/<0,-700>[{\mathbb A}_1 `{\mathbb B}_1;U_1]

\morphism(1000,700)|r|/@{>}|{\bb}/<0,-700>[{\mathbb A}_1 `{\mathbb B}_1;U_1]

\morphism(500,1050)|r|/@{>}|{\bb}/<0,-700>[{\mathbb A}_2`{\mathbb B}_2;U_2]

\place(250,550)[\scriptstyle \eta_1]

\place(750,550)[\scriptstyle \mu]

\place(1500,500)[=]

\square(2000,0)/=`@{>}|{\bb}`@{>}|{\bb}`=/<1000,700>[{\mathbb A}_1 `{\mathbb A}_1`{\mathbb B}_1 `{\mathbb B}_1;`U_1`U_1`]

\Atriangle(2000,700)/<-`>`=/<500,350>[{\mathbb A}_2`{\mathbb A}_1`{\mathbb A}_1;\id_1`m`]

\morphism(2500,950)|r|/=>/<0,-175>[`;\scriptstyle \lambda]

\place(2500,350)[\scriptstyle 1_{U_1}]

\place(-700,650)[\ \ ]
\place(3700,650)[(5)]
\efig
$$

$$
\bfig\scalefactor{.8}

\Atriangle/<-`>`=/<500,350>[{\mathbb B}_2 `{\mathbb B}_1 `{\mathbb B}_1;\id_2`n`]

\morphism(500,250)|r|/=>/<0,-175>[`;\scriptstyle \ \rho]

\Atriangle(0,700)/<-`>`/<500,350>[{\mathbb A}_2`{\mathbb A}_1`{\mathbb A}_2;\id_2`m`]

\morphism(0,700)|l|/@{>}|{\bb}/<0,-700>[{\mathbb A}_1 `{\mathbb B}_1;U_1]

\morphism(1000,700)|r|/@{>}|{\bb}/<0,-700>[{\mathbb A}_1 `{\mathbb B}_1;U_1]

\morphism(500,1050)|r|/@{>}|{\bb}/<0,-700>[{\mathbb A}_2`{\mathbb B}_2;U_2]

\place(250,550)[\scriptstyle \eta_2]

\place(750,550)[\scriptstyle \mu]

\place(1500,500)[=]

\square(2000,0)/=`@{>}|{\bb}`@{>}|{\bb}`=/<1000,700>[{\mathbb A}_1 `{\mathbb A}_1`{\mathbb B}_1 `{\mathbb B}_1;`U_1`U_1`]

\Atriangle(2000,700)/<-`>`=/<500,350>[{\mathbb A}_2`{\mathbb A}_1`{\mathbb A}_1;\id_2`m`]


\place(2500,350)[\scriptstyle 1_{U_1}]

\place(-700,650)[\ \ ]
\place(3700,650)[(6)]
\efig
$$
\end{definition}

\begin{proposition} A vertical morphism of horizontal intercategories is the same as a colax-colax morphism of double pseudocategories.

\end{proposition}

\begin{proof} (Sketch) We simply point out where the structural morphisms (dual of (1)-(4) in Section \ref{mor}) come from and which can be seen to go in the right directions. We omit the routine verification of (the dual) conditions (5)-(14).

Because $ U_2 $ is colax, we get
$$
U(\Id_f) \to \Id_{Uf} \mbox{\ \ and\ \ } U \left(\frac{\alpha}{\ov{\alpha}} \right) \to \frac{U\alpha}{U\ov{\alpha}} .
$$
The cells $ \eta $ and $\mu $ give
$$
\eta_v : U (\id_v) \to \id_{Uv} \mbox{\ \ and\ \ } \mu_{\alpha, \beta} : U (\alpha |\beta) \to U \alpha | U \beta .
$$

\end{proof}

We can also define double cells between horizontal lax (resp.~horizontal colax) and vertical morphisms. (Cells between horizontal lax and horizontal colax have already been studied in Section \ref{pscat}.) Let $ {\ic F} $ and $ {\ic G} $ be horizontal lax morphisms, and $ {\ic U} $ and $ {\ic V} $ vertical morphisms as in
$$
\bfig

\square/>`@{>}|{\bb}`@{>}|{\bb}`>/[{\ic A} `{\ic B} `{\ic C} `{\ic D};{\ic F} `{\ic U} `{\ic V} `{\ic G}]

\place(250,250)[\scriptstyle \pi]

\efig
$$
A cell $ \pi $ as above consists of three cells
$$
\bfig

\square/>`@{>}|{\bb}`@{>}|{\bb}`>/[{\mathbb A}_i `{\mathbb B}_i `{\mathbb C}_i `{\mathbb D}_i;F_i `U_i `V_i `G_i]

\place(250,250)[\scriptstyle \pi_i]

\efig
$$
preserved by domain and codomain (so $ \pi_2 $ is determined by $ \pi_0 $ and $\pi_1 $)
$$
\bfig

\square/>`@{>}|{\bb}`@{>}|{\bb}`>/[{\mathbb A}_1 `{\mathbb A}_0 `{\mathbb C}_1 `{\mathbb C}_0;
        \partial_i `U_1 `U_0 `\partial_i]

\place(250,250)[\scriptstyle \delta_i]

\square(500,0)/>`@{>}|{\bb}`@{>}|{\bb}`>/[{\mathbb A}_0 `{\mathbb B}_0 `{\mathbb C}_0 `{\mathbb D}_0;
        F_0 ``V_0`G_0]

\place(750,250)[\scriptstyle \pi_0]

\square(0,500)/>`=``>/[{\mathbb A}_1 `{\mathbb B}_1 `{\mathbb A}_1 `{\mathbb A}_0;
     F_1 ```]

\square(500,500)/>``=`>/[{\mathbb B}_1 `{\mathbb B}_0 `{\mathbb A}_0 `{\mathbb B}_0;
      \partial_i ```]

\place(500,750)[=]

\place(1250,500)[=]

\square(1500,0)/>`=``>/[{\mathbb C}_1 `{\mathbb D}_1  `{\mathbb C}_1 `{\mathbb C}_0;```\partial_i]

\square(2000,0)/>``=`>/[{\mathbb D}_1  `{\mathbb D}_0 `{\mathbb C}_0 `{\mathbb D}_0;```G_0]

\place(2000,250)[=]

\square(1500,500)/>`@{>}|{\bb}`@{>}|{\bb}`>/[{\mathbb A}_1  `{\mathbb B}_1 `{\mathbb C}_1 `{\mathbb D}_1;F_1`U_1 `V_1`G_1]

\square(2000,500)/>`@{>}|{\bb}`@{>}|{\bb}`>/[{\mathbb B}_1 `{\mathbb B}_0 `{\mathbb D}_1 `{\mathbb D}_0;\partial_i ``V_0`\partial_i]

\place(1750,750)[\scriptstyle \pi_1]

\place(2250,750)[\scriptstyle \delta_i]

\efig
$$
and respecting the structural elements of the morphisms in question
$$
\bfig

\square/>`@{>}|{\bb}`@{>}|{\bb}`>/[{\mathbb A}_0 `{\mathbb A}_1 `{\mathbb C}_0 `{\mathbb C}_1;`U_0`U_1`\id]

\place(250,250)[\scriptstyle \eta]

\square(500,0)/>`@{>}|{\bb}`@{>}|{\bb}`>/[{\mathbb A}_1  `{\mathbb B}_1 `{\mathbb C}_1 `{\mathbb D}_1;``V_1`G_1]

\place(750,250)[\scriptstyle \pi_1]

\square(0,500)/>`=``>/[{\mathbb A}_0 `{\mathbb B}_0 `{\mathbb A}_0 `{\mathbb A}_1;F_0 ```\id]

\square(500,500)/>``=`>/[{\mathbb B}_0 `{\mathbb B}_1 `{\mathbb A}_1 `{\mathbb B}_1;\id ```F_1]

\place(500,750)[\phi_0]

\place(1250,500)[=]

\square(1500,0)/>`=``>/[{\mathbb C}_0 `{\mathbb D}_0 `{\mathbb C}_0 `{\mathbb C}_1;```\id]

\square(2000,0)/>``=`>/[{\mathbb D}_0 `{\mathbb D}_1 `{\mathbb C}_1 `{\mathbb D}_1;```G_1]

\place(2000,250)[\gamma_0]

\square(1500,500)/>`@{>}|{\bb}`@{>}|{\bb}`>/[{\mathbb A}_0 `{\mathbb B}_0 `{\mathbb C}_0 `{\mathbb D}_0;F_0 `U_0`V_0`G_0]

\square(2000,500)/>`@{>}|{\bb}`@{>}|{\bb}`>/[{\mathbb B}_0 `{\mathbb B}_1 `{\mathbb D}_0 `{\mathbb D}_1;\id ``V_1`\id]

\place(1750,750)[\scriptstyle \pi_0]

\place(2250,750)[\scriptstyle \epsilon]

\efig
$$

$$
\bfig

\square/>`@{>}|{\bb}`@{>}|{\bb}`>/[{\mathbb A}_2 `{\mathbb A}_1 `{\mathbb C}_2 `{\mathbb C}_1;`U_2``m]

\place(250,250)[\scriptstyle \mu]

\square(500,0)/>`@{>}|{\bb}`@{>}|{\bb}`>/[{\mathbb A}_1  `{\mathbb B}_1 `{\mathbb C}_1 `{\mathbb D}_1;`U_1`V_1`G_1]

\place(750,250)[\scriptstyle \pi_1]

\square(0,500)/>`=``>/[{\mathbb A}_2 `{\mathbb B}_2 `{\mathbb A}_2 `{\mathbb A}_1;F_2 ```m]

\square(500,500)/>``=`>/[{\mathbb B}_2 `{\mathbb B}_1 `{\mathbb A}_1 `{\mathbb B}_1;m ```F_1]

\place(500,750)[\scriptstyle \phi_2]

\place(1250,500)[=]

\square(1500,0)/>`=``>/[{\mathbb C}_2 `{\mathbb D}_2 `{\mathbb C}_2 `{\mathbb C}_1;```m]

\square(2000,0)/>``=`>/[{\mathbb D}_2 `{\mathbb D}_1 `{\mathbb C}_1 `{\mathbb D}_1;```G_1]

\place(2000,250)[\scriptstyle \gamma_0]

\square(1500,500)/>`@{>}|{\bb}`@{>}|{\bb}`>/[{\mathbb A}_2 `{\mathbb B}_2 `{\mathbb C}_2 `{\mathbb D}_2;F_2 `U_2`V_2`G_2]

\square(2000,500)/>`@{>}|{\bb}`@{>}|{\bb}`>/[{\mathbb B}_2 `{\mathbb B}_1 `{\mathbb D}_2 `{\mathbb D}_1;m ``V_1`m]

\place(1750,750)[\scriptstyle \pi_2]

\place(2250,750)[\scriptstyle \gamma]

\efig
$$

It is routine to check that these cells compose horizontally and vertically giving a strict double category structure.

We leave it to the reader to formulate the similar notion of cell between horizontal colax and vertical morphisms.

Finally, we define cubes
$$
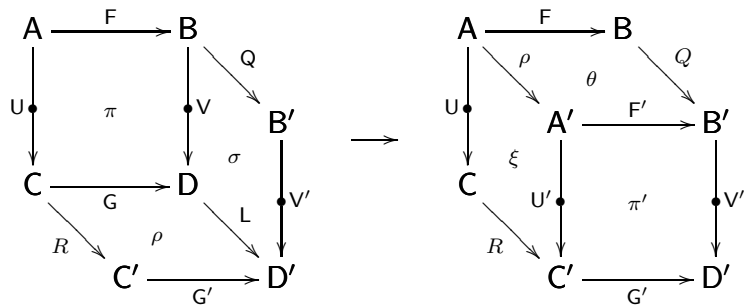
\bfig

\square(0,300)/>`@{>}|{\bb}`@{>}|{\bb}`>/[{\ic A} `{\ic B} `{\ic C} `{\ic D};{\ic F} `{\ic U} `{\ic V} `{\ic G}]

\square(300,0)/``@{>}|{\bb}`>/[`{\ic B'}`{\ic C'}`{\ic D'};` `{\ic V'} `{\ic G'}]

\morphism(500,800)<300,-300>[{\ic B} `{\ic B'};{\ic Q}]

\morphism(0,300)|b|<300,-300>[{\ic C}`{\ic C'};R]

\morphism(500,300)<300,-300>[{\ic D} `{\ic D'};{\ic L}]

\place(250,550)[\scriptstyle \pi]

\place(400,150)[\scriptstyle \rho]

\place(650,400)[\scriptstyle \sigma]

\place(1100,450)[\to]

\square(1400,300)/>`@{>}|{\bb}``/[{\ic A} ` {\ic B} `{\ic C} `;{\ic F}`{\ic U} ` `]

\square(1700,0)/>`@{>}|{\bb}`@{>}|{\bb}`>/[{\ic A'} ` {\ic B'} `{\ic C'} `{\ic D'}; {\ic F'}` {\ic U'}`{\ic V'}`{\ic G'}]

\morphism(1400,800)|a|/>/<300,-300>[{\ic A}`{\ic A'};\rho]

\morphism(1900,800)/>/<300,-300>[{\ic B} ` {\ic B'};Q]

\morphism(1400,300)|b|/>/<300,-300>[{\ic C}`{\ic C'};R]

\place(1950,250)[\scriptstyle \pi']

\place(1550,400)[\scriptstyle \xi]

\place(1800,650)[\scriptstyle \theta]

\efig
$$
There is at most one and there is one if and only if the cube commutes. Each of the faces consists of three cells and to say that the cube commutes means that for $ i = 0, 1, 2 $ we have
$$
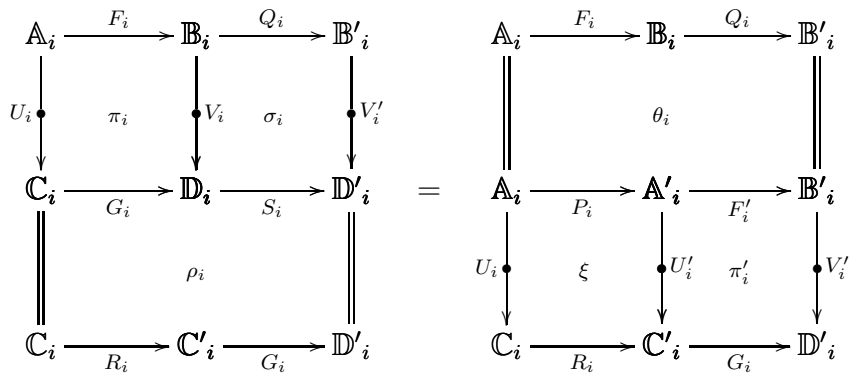
\bfig

\square/>`=``>/[{\mathbb C}_i `{\mathbb D}_i `{\mathbb C}_i `{\mathbb C'}_i;```R_i]

\square(500,0)/>``=`>/[{\mathbb D}_i `{\mathbb D'}_i `{\mathbb C'}_i `{\mathbb D'}_i;```G_i]

\place(500,250)[\scriptstyle \rho_i]

\square(0,500)/>`@{>}|{\bb}`@{>}|{\bb}`>/[{\mathbb A}_i `{\mathbb B}_i `{\mathbb C}_i `{\mathbb D}_i;F_i `U_i`V_i`G_i]

\place(250,750)[\scriptstyle \pi_i]

\square(500,500)/>`@{>}|{\bb}`@{>}|{\bb}`>/[{\mathbb B}_i `{\mathbb B'}_i `{\mathbb D}_i `{\mathbb D'}_i;Q_i``V'_i`S_i]

\place(750,750)[\scriptstyle \sigma_i]

\place(1250,500)[=]

\square(1500,0)/>`@{>}|{\bb}`@{>}|{\bb}`>/[{\mathbb A}_i `{\mathbb A'}_i `{\mathbb C}_i `{\mathbb C'}_i;`U_i`U'_i`R_i]

\square(2000,0)/>`@{>}|{\bb}`@{>}|{\bb}`>/[{\mathbb A'}_i `{\mathbb B'}_i`{\mathbb C'}_i `{\mathbb D'}_i;``V'_i`G_i]

\place(1750,250)[\scriptstyle \xi]

\place(2250,250)[\scriptstyle \pi'_i]

\square(1500,500)/>`=``>/[{\mathbb A}_i `{\mathbb B}_i `{\mathbb A}_i `{\mathbb A'}_i;F_i```P_i]

\square(2000,500)/>``=`>/[{\mathbb B}_i `{\mathbb B'}_i`{\mathbb A'}_i `{\mathbb B'}_i;Q_i```F'_i]

\place(2000,750)[\scriptstyle \theta_i]

\efig
$$

In this way we get a strict triple category ${\ic HICat} $ whose objects are horizontal pseudocategories in $ \Doub $. The transversal morphisms are the horizontal colax morphisms, the horizontal morphisms of ${\ic HICat} $ are the horizontal lax morphisms and the vertical morphisms are what we have just called vertical. The cells and cubes are as defined above. There is a similar construction giving a strict triple category $ {\ic VICat} $ of vertical pseudocategories in $ \Doub $. 

The following theorem expresses that our three presentations of intercategory produce the same structure.

\begin{theorem}\label{lastth} The triple categories ${\ic ICat} $, ${\ic HICat} $ and $ {\ic VICat} $ are isomorphic.

\end{theorem}


\begin{references*}

\bibitem{AM} M.\ Aguiar and S.\ Mahajan, Monoidal functors, species and Hopf algebras, CRM Monograph Series 29, American Mathematical Society, Providence, RI, 2010.

\bibitem{laxtransf} J.\ B\'enabou, Introduction to bicategories, Reports of the Midwest Category Seminar, Lecture Notes in Mathematics, Vol. 47, Springer, Berlin 1967, pp.\ 1-77.

\bibitem{BCZ} G.\ B\"{o}hm, Y.\ Chen, L.\ Zhang,  On Hopf monoids in duoidal categories,  Journal of Algebra,
Vol.~394, 2013, pp.~139-172.

\bibitem{BS} T.\ Booker and R.\ Street, Tannaka duality and convolution for duoidal categories, Theory Appl. Categ. 28 (2013), No.\ 6, pp.\ 166-205.

\bibitem{E}  C.\ Ehresmann, Cat\'egories structur\'ees, Ann. Sci. Ecole Norm. Sup. 80 (1963), 349-425.

\bibitem{GPS} R.\ Gordon, A.\ J.\ Power, and R.\ Street, Coherence for Tricategories, Mem.  Amer. Math. Soc., 1995, Vol.\ 117, no.\ 558.

\bibitem{CC} M.\ Grandis, Higher cospans and weak cubical categories (Cospans in Algebraic Topology, I), Theory Appl. Categ. 18 (2007), no.\ 12, pp.\ 321-347.

\bibitem{GP99} M.\ Grandis and R.\ Par\'e, Limits in double categories, Cahiers Topologie G\'eom.\ Diff\'erentielle Cat\'eg.\ 40 (1999), no.\ 3, pp.\ 162-220.

\bibitem{GP} M.\ Grandis and R.\ Par\'e, Adjoint for double categories, Cahiers Topologie G\'eom.~Diff\'erentielle Cat\'eg.\ 45 (2004), pp.\ 193-240.

\bibitem{PartII} M.\ Grandis and R.\ Par\'e, Intercategories: a framework for three-dimensional category theory, in preparation.

\bibitem{GP-Multiple} M.\ Grandis and R.\ Par\'e, An introduction to multiple categories (On weak and lax multiple categories, I), in preparation.

\bibitem{GP2} M.\ Grandis and R.\ Par\'e, Limits in multiple categories (On weak and lax multiple categories, II), in preparation.

\bibitem{L} S.\ Lack, Icons, Applied Categorical Structures 18(3), 2010, pp.\ 289-307.

\bibitem{PC} N. Martins-Ferreira, Pseudo-Categories, Journal of Homotopy and Related Structures, vol.\ 1(1), 2006, pp.\ 47-78.

\bibitem{V} D.\ Verity, Enriched categories, internal categories and change of base, Reprints in Theory and Applications of Categories, no.\ 20, 2011, pp.\ 1-266.

\end{references*}


\end{document}